\documentclass[a4paper,11pt]{amsart}

\usepackage{geometry}
\usepackage{color}

\usepackage[colorlinks=true,urlcolor=blue,citecolor=red,linkcolor=blue,linktocpage,pdfpagelabels,bookmarksnumbered,bookmarksopen]{hyperref}
\usepackage[hyperpageref]{backref}

\usepackage{todonotes}

\usepackage{amssymb}

\newcommand{\Haus}{\mathcal{H}}

\newcommand{\om}{\Omega}

\newtheorem{thm}{Theorem}[section]
\newtheorem{cor}[thm]{Corollary}
\newtheorem{lem}[thm]{Lemma}
\newtheorem{prop}[thm]{Proposition}
\theoremstyle{definition}
\newtheorem{defn}[thm]{Definition}
\newtheorem{rem}[thm]{Remark}
\newtheorem{ex}[thm]{Example}

\newcommand{\dist}{\operatorname{d}}
\newcommand{\norm}[1]{\left\Vert#1\right\Vert}

\newcommand{\R}{\mathbb R}
\newcommand{\N}{\mathbb N}
\newcommand{\eps}{\varepsilon}

\newcommand{\matrice}[4]
{\left(\begin{array}{cc} #1&#2
\\
#3&#4
\end{array} \right) }

\newcommand{\weak }{\, -\!\!\!\!-\!\!\!\rightharpoonup}
\newcommand{\weakstar }{ \overset{\, *_{\phantom{|}}}{{\smash{\weak }}\, } }

\title[Selection of solutions for implicit systems of PDE]{Variational methods for the selection of solutions to an implicit system of PDE}

\author[Croce]{Gisella Croce}
\author[Pisante]{Giovanni Pisante}

\address[G. Croce]{ Normandie Univ, UNIHAVRE, LMAH, FR-CNRS-3335, 76600 Le Havre, France}
\email{gisella.croce@univ-lehavre.fr}

\address[G. Pisante]{Dipartimento di Matematica e Fisica
\newline\indent 
Universit\`a degli Studi della Campania ``Luigi Vanvitelli''
\newline\indent 
Viale Lincoln 5, 81100 Caserta, Italy}
\email{giovanni.pisante@unina2.it}

\subjclass[2010]{34A60, 35A15, 35F30, 49J40, 49Q15}
\keywords{almost everywhere solutions, orthogonal group, functions of bounded
variation, direct methods of the calculus of variations.}

\begin{document}

\numberwithin{equation}{section}

\maketitle

\begin{abstract} 
We consider the vectorial system
\[
\begin{cases}
	Du \in \mathcal{O}(2), & \mbox{a.e. in}\;\Omega, \\
	u=0, & \mbox{on} \;\partial \Omega,
\end{cases}
\]
where $\Omega$ is a subset of $\R^2$, $u:\Omega\to \R^2$ and $\mathcal{O}(2)$ is the orthogonal group of $\R^2$. 
We provide a variational method to select, among the infinitely many solutions, the ones that minimize an appropriate weighted measure of some set of singularities of the gradient.
\end{abstract}

\section{Introduction}
\label{sec:intro}

In the last decades a great effort has been devoted to the study of nonlinear systems of partial differential equations of implicit type. Given an open bounded subset  $\Omega\subset \R^n$, let
$F_i: \R^{N\times n}\to \R$, $i\in\{1,\dots,m\}$ and $\varphi: \overline{\Omega}\to\mathbb{R}^{N}$, the prototype problem can be written as 
\begin{equation}
\label{general differential inclusion}
\begin{cases}
	F_i(Du)=0, & \mbox{a.e. in}\;\Omega\,,
\\
u=\varphi, &  \mbox{on}\,\,\partial\Omega,
\end{cases}
\end{equation}
or equivalently as the differential inclusion 
\begin{equation}
	\label{inclusion}
\begin{cases}
Du\in E,& \mbox{a.e. in}\, \,\Omega\,,
\\
u=\varphi,&  \mbox{on}\,\,\partial\Omega,
\end{cases}
\end{equation}
where 
\[
E:= \big\{ \xi \in \R^{N\times n} \;:\; F_i(\xi)=0 \,,\; i\in \{1,\dots,m\} \big\} .
\]
Different and quite general methods have been developed to prove the existence of almost everywhere $W^{1,\infty}(\Omega,\R^N)$  solutions to \eqref{general differential inclusion} under suitable mild regularity assumptions on the functions $F_i$ and $\varphi$. 
In the scalar case,  i.e. $N=1$,  we can for instance rely on the viscosity method, initiated by Crandall and P.-L. Lions \cite{Crandall:1983hp}, the pyramidal construction by Cellina \cite{Cellina:1993jf}, on the Baire category method introduced by Cellina in \cite{Cellina:1993jf,Cellina:1993bw} and later developed by Dacorogna and Marcellini in \cite{Dacorogna:1997fv} (see also the monograph \cite{Dacorogna:2012eu} and the references therein) and also on the Gromov integration approach developed by M\"uller and Sverak in \cite{Muller:1995tp, Muller:1999bl, Muller:2003db}. The last two approaches are suitable to be applied also in the vectorial setting, i.e. for $N>1$. The pyramidal construction, the Baire category method and the Gromov integration approach are not constructive and usually, when they can be applied, provide the existence of infinitely many solutions. Thus the question of selecting a preferred solution among them raised.

To underline the difficulties one can encounter, we first discuss the scalar case, $N=1$. A natural idea would be 
to use  the theory of viscosity solutions. This would serve as a perfect selection principle, providing uniqueness as well as explicit formulas for the solution. Nevertheless its applicability is  limited. Indeed, the existence of a viscosity solution can be proved only under quite strict compatibility conditions between the geometry of $\Omega$ and the set $E$ (cfr. \cite{Cardaliaguet:1999dj} and \cite{Pis04-001} for a complete analysis). 
A more general approach has been proposed by Cellina 
in \cite{Cellina:1993jf}
. Under the hypothesis that the boundary datum $\varphi$ is affine, his construction gives an explicit solution to (\ref{general differential inclusion}) in a special domain $P$ related to the functions $F_i$. For example, assuming $\varphi=0$, if $0$ can be written as a convex combination of a finite number of matrices $\xi_i$, $ i\in I$ with $I:=\{1,\dots, l\}$, belonging to  $E$, then
the {\it pyramid} defined by 
\[
p(x):= r-\max \big\{\langle \xi_i, x-x_0\rangle, i\in I\big\}\,,
\]
for $r>0$ and $x_0\in \R^n$,
is a $W^{1,\infty}_0(P)$ solution of \eqref{general differential inclusion}
in the domain
\[
P:=\big\{x\in \R^n: r-\max\{\langle\xi_i, x-x_0\rangle, i\in I\}\geq 0
\big\}.
\] 
Observe that the pyramid is an affine piecewise function whose gradient takes only a finite number of values, $\{\xi_i\}_{i\in I}$. Using as building blocks the rescaled pyramids, it is then possible to construct a solution (and actually infinitely many) of \eqref{general differential inclusion} in a general domain $\Omega$ by a Vitali covering. It is worth observing that, unless $\Omega$ has a very special geometry, imposing the boundary condition forces the solutions to have a fractal behavior near the boundary.
An explicit Vitali covering made up of sets where a viscosity solution exists has been proposed in  \cite{Dacorogna:2004gc}. 

In the recent literature, inspired by the Cellina's construction, some selection criteria have been proposed to somehow minimize the irregularities of the solutions, and taking into account  their fractal behavior. Most of the results are restricted to the the case where $E$ is a finite set. 
In this framework, in \cite{Champion:2006uh} and \cite{Croce:kw}, the attention has been focused on the system of eikonal equations in dimension $n=2$: 
\begin{equation}\label{SEE}
\begin{cases}
	\displaystyle\left|\frac{\partial u}{\partial x_i}\right|=1,\,
&i=1, 2,\,\, \mbox{a.e. in}\;\Omega,
\\
u=0, & \mbox{on} \; \partial \Omega.
\end{cases}
\end{equation}
Since a viscosity solution exists only in rectangles whose sides are parallel to $x_2=\pm x_1$ (cfr. \cite{Pis04-001}), for quite general domains, we proposed a variational argument to select the, roughly speaking, "most regular" solutions to (\ref{SEE}), through the minimization of the set of the irregularities of their gradient. More precisely we considered the  functional
$$
\mathcal{\mathcal{D}}(v)=\sum_{i=1}^2\int_{\Omega}H(d_1(x,\partial \Omega))d\left|Dv_{x_i}\right|\,,
$$
where the lower script denotes partial differentiation, $H: \R^+\to\R^+$ is a continuous increasing function such that 
\[
\int_0^1 \frac{H(t)}{t}dt<+\infty
\]
and $d_1$ is the distance in the $l^1$ norm.
The fractal behavior of the singular set of a solution could be spread also far from the boundary of $\Omega$. 
Nevertheless, these pathological solutions should not be considered as good candidates for our selection principle. This is why we considered this functional over the set of solutions $v$ to (\ref{SEE}) such that $v_{x_i}\in SBV_{loc}(\Omega), i=1, 2$. The weight function $H(d_1(\cdot,\partial \Omega))$  has been introduced to deal with  the general fractal behavior of the solutions near the boundary. We observe that  the knowledge of the pyramidal construction of Cellina is twofold for our result. 
 On one hand, the analysis of the regularities of its gradient has inspired the choice of the energy functional and on the other one, it is a key ingredient in the proof of the existence of a minimizer for $\mathcal{D}$.
  Indeed it allows to give a meaning to the variational problem providing an explicit solution $u$ with bounded energy, i.e. with $\mathcal D(u)< \infty$. 
 
 Our aim in this paper is to extend this variational approach to the selection of solutions to a vectorial problem. Passing from the scalar to the vectorial case, several difficulties come into play.  For example,  there is not a suitable notion of viscosity solution neither a general way of constructing a simple pyramid in the spirit of Cellina's works.

An explicit construction of solutions has been provided for the problem 
\[
\begin{cases}
	Du \in \mathcal{O}(n), & \mbox{a.e. in}\;\Omega, \\
	u=0, & \mbox{on} \;\partial \Omega,
\end{cases}
\]
with $\mathcal{O}(n)$ denoting the orthogonal group of matrices of $\R^n$ in \cite{Cellina:1995uv} and \cite{Dacorogna:2008ih}. In both papers the authors exhibit an explicit solution in a square and a cube, in the spirit of the Cellina's pyramid of the scalar setting, but far from being so simple. In particular, these so called {\it vectorial pyramids}, $p_v$, are again maps whose gradient takes only a finite number of values, $\{\xi_i\}_{i\in I}\subset \R^{n\times n}$, but for any $i\in I$ the set $\Omega^p_i:=\{ Dp^v(x)=\xi_i \}$ is disconnected, with infinitely many connected components. Therefore the solution has a fractal behavior at the boundary. This is an important difference with respect to the scalar case. 
Indeed if one uses a Vitali covering argument to define a solution in a general domain $\Omega$, by patching the rescaled vectorial pyramids, he obtains  a solution with a fractal behavior of its singular set also far from the boundary of $\Omega$ and not only near the boundary. We do not know if in the case of a general domain there exists a solution without fractal behavior far from the boundary of $\Omega$ and only at the boundary, as in the case of the square. Therefore a selection principle should take into account this possibility.

The present study stems from the analysis of the properties 
of the vectorial pyramid constructed in \cite{Dacorogna:2008ih} as a special solution to the Dirichlet problem
\begin{equation}\label{OP}
\begin{cases}
	Du \in \mathcal{O}(2), & \mbox{a.e. in}\;\Omega, \\
	u=0, & \mbox{on} \;\partial \Omega.
\end{cases}
\end{equation}
As for the construction in the scalar case, only a finite subset  $E\subset \mathcal{O}(2)$
has been considered, namely 
\begin{equation}\label{eq:set-E}
E:= 
\left\{
\pm 
\matrice{1}{0}{0}{1},
\pm 
\matrice{1}{0}{0}{-1},
\pm
\matrice{0}{1}{1}{0},
\pm 
\matrice{0}{1}{-1}{0}
\right\}\,.
\end{equation}
In \cite{Dacorogna:2008ih}, the authors
constructed an explicit  
solution $p_v$
in $\Omega=(-2,2)\times (-2,2)$ (cfr. Section \ref{subsec:DMP-construction})
 of 
\begin{equation}\label{OPGG}
\begin{cases}
	Du \in E, & \mbox{a.e. in}\;\Omega, \\
	u=0, & \mbox{on} \;\partial \Omega.
\end{cases}
\end{equation}

Here we propose a variational criterion to select a solution of  problem \eqref{OPGG} in the spirit of \cite{Croce:kw}. 
The vectorial pyramid $p_v$ will play the same  role as the pyramid $p$ of the scalar case. 
As already observed, the main source of difficulties, that also characterize the main novelty of the paper, is the necessity to take into account the possibility of a fractal behavior of the singularities of a given solution $u$ in $\Omega$. To this aim, we define $\Sigma^u_{\infty}$ to be, roughly speaking, the set where the singularities of $u$ accumulate (cfr. Definition \ref{def:sigmainfty}) and we consider the energy functional 
\[
\mathcal{F}(u)=\int_{\Omega}\dist(x,\partial \Omega) \chi_{\Sigma^u_{\infty}}d\mathcal{H}^1+\sum_{i,j=1}^2\int_{\Omega}\big(\dist(x,\Sigma^u_{\infty})\big)^{\alpha} \,d|Du^j_{x_i}|\,,
\]
where $\alpha>0$ will depend on the geometry of the domain $\Omega$ (see Definitions \ref{defn:compatible-triang} and   \ref{defn:compatibledomain}). The choice of the functional $\mathcal F$ has been motivated by the idea that solutions with a small  $\Sigma^u_\infty$ should be preferred. The role of the first term of $\mathcal{F}$  is  to discard pathological solutions with $\Sigma^u_\infty$ not locally bounded with respect to the $\Haus^1$-measure and to control its fractal behavior at the boundary. The second term, roughly speaking, minimizes the singularities of $Du$ in $\Omega$. The weight depending on the distance from $\Sigma^u_\infty$ is actually necessary since in general we cannot prevent the singularities of $Du$ to accumulate near $\Sigma^u_\infty$ and to be supported in a $\Haus^1$-dimensional set with infinite lenght.
Let us observe also that, in the case when the fractal behavior of $Du$ is concentrated only near the boundary of $\Omega$, (e.g. if we consider the vectorial pyramid $p_v$ in the square $(-a,a)\times (a,a)$ as in \cite{Dacorogna:2008ih}), the first term in $\mathcal F$ is identically zero and the second term performs the selection by choosing the solutions that minimize the weighted length of the jumps of  the gradient. 

We assume that $\Omega$ is a compatible domain, according to  Definition \ref{defn:compatibledomain}. 
As we will see, this hypothesis is important to prove that our variational problem is well posed, in a suitable subclass of solutions to (\ref{OPGG}), that we denote by $\mathcal{S}_c$.
We prove the existence of a minimizer of $\mathcal{F}$ for the maps $u$ in 
$\mathcal{S}_c$, for which $\Sigma^u_\infty$ has some "good" properties. More precisely, under suitable assumptions about the connectedness of $\Sigma^u_\infty$ and on its $\mathcal{H}^1$-measure (see Theorem \ref{thm:main}), we can ensure compactness and semicontinuity of the functional $\mathcal{F}$.

The paper is organized as follows. In the next section we fix the notations, recall some preliminaries results of geometric measure theory needed in the sequel and we describe the vectorial pyramidal construction in the square of \cite{Dacorogna:2008ih}. In Section \ref{sec:PropertiesSolutions} we study some  properties of Lipschitz vector valued maps whose gradient takes only a finite number of values. For a given $u$ of such type we define the set $\Sigma^u_\infty$ and we study some of its properties. Section
\ref{sec:energybound} is devoted to the study of the compactness and semicontinuity properties of the functional $\mathcal F$. 
 In section \ref{sec:funzionalefinito} we present some quite general classes of domains where our selection principle, based on minimization of the functional $\mathcal{F}$ can be applied.

\section{Preliminaries}\label{Sec:Preliminaries}

\subsection{Notations}\label{subsec:notation}
Throughout this paper $\mathcal{L}^n$ and $\Haus^k$ we denote the $n$-dimensional Le\-bes\-gue measure and the $k$-dimensional Hausdorff measure. Open balls in $\R^n$ centered at $x$ with radius $r$ will be usually denoted with $B(x,r)$, $\omega_n:=\mathcal{L}^n(B(0,1))$ is the $n$-dimensional Lebesgue measure of $B(0,1)$. Given a set $S\subset \R^n$ and $\rho>0$, we denote by $I_{\rho}(S)$ the open $\rho$ neighborhood of $S$, that is, 
\[
I_{\rho}(S):=\{x\in \R^n\;:\; \dist(x,S)<\rho\}.
\] 
Clearly $I_r( \{x\})=B(x,r)$ and we will write simply $I_r(x)$ for this set. 
We denote by $\chi_S$ the characteristic function of $S$, that is, the function equal to 1 if $x\in S$ and 0 otherwise. The complemet of $S$, $\R^{n}\setminus S$ will be denoted by $[S]^{c}$.

We denote the distributional gradient of a map $u$ with $Du$. For a vector valued function $u:\Omega\to \R^2$ we use upper indexes to denote its components, $u=(u^1,u^2)$ and we adopt the self explanatory lower scripts notation for weak derivatives, whenever they are well defined, $u_{x_i}=(u^1_{x_i},u^2_{x_i})$. If $\mu$ is a measure, we denote by $|\mu|$ and $\operatorname{supp}\mu$ its total variation and its support respectively.  

\subsection{Minkowski content}
Here we recall some basic properties of the intrinsic definition of area due to H. Minkowski, mainly introduced for compact sets and named after him as {\it Minkowski content}. For our purpose, we confine ourself to the one-dimensional content in the two-dimensional Euclidean setting. For the general theory, for the detailed proofs of the result of this subsection and for further applications, the interested reader may refer to \cite[3.3]{Krantz:2012vr}, \cite[2.13]{AmbFusPal00-000}, and \cite[3.2.37]{Fed69-000}. 

\begin{defn}[Upper and lower Minkowski content]
\label{minkowski}
Let $S\subset \R^2$ be a closed set. The {\it upper and lower 1-dimensional Minkowski contents} $\mathcal M^*(S)$, $\mathcal M_*(S)$ of $S$ are respectively defined by
\[
\mathcal M^*(S):= \limsup_{\rho \downarrow 0} \frac{\mathcal L^2\left(I_\rho(S)\right)}{2 \rho}\;,\;\;\;\mathcal M_*(S):= \liminf_{\rho \downarrow 0} \frac{\mathcal L^2\left(I_\rho(S)\right)}{2 \rho}.
\]
If $\mathcal M^*(S)=\mathcal M_*(S)$, this quantity, denoted by $\mathcal M(S)$, is called Minkowski content of $S$. 
\end{defn}

As we will see, in the proof of compactness of minimizing sequences of our functional (see Theorem \ref{thm:compact}) we will need an upper bound for $\mathcal M(S)$ in terms of $\mathcal H(S)$. These type of bounds are in general not true. For any $\mathcal H^1$-rectifiable closed set $S$, one has $\mathcal H^1(S) \leq \mathcal M^*(S)$, but an upper bound for $\mathcal M^*(S)$ in terms of $\mathcal H^1(S)$ is a more delicate issue. Indeed the sole rectifiability is not sufficient, nevertheless, an additional assumption of density lower bound turns out to be sufficient for the desired upper bound, leading to the following theorem. 

\begin{thm}
	\label{Thm:minkowskylowerbound}
	Let $S\subset \R^2$ be a countably $\mathcal H^1$-rectifiable compact set. Assume that there exist $\gamma >0$ and a Radon measure $\nu$ in $\R^2$ absolutely continuous with respect to $\mathcal H^1$ such that
	\begin{equation}
		\label{eq:densityLB}
		\nu \left( B(x,\rho)\right) \geq \gamma \rho \;\;\;\;\; \forall\, x \in S \;,\;\;\rho \in (0,1).
	\end{equation} 
			Then $\mathcal M(S)=\mathcal H^1(S)$.
\end{thm}

\subsection{Hausdorff metric}\label{subsec:hausdorff}

Let $\Omega \subset \R^2$ be open and bounded.
Let $\mathcal{K}(\overline{\Omega})$ be the set of all compact subsets of $\overline{\Omega}$ and  
$\mathcal{K}^f(\overline{\Omega})\subset \mathcal{K}(\overline{\Omega})$ be composed by the subsets that are connected and 
with finite $\mathcal{H}^1$-measure. We recall that the {\it Hausdorff distance} between two sets $K_1$ and $K_2$ in $\mathcal K(\overline \Omega)$ is defined by 
\[
d_{\mathcal{H}}(K_1,K_2):= \max \left\{ \sup_{x\in K_1} \dist(x,K_2), \sup_{x\in K_2} \dist(x,K_1) \right\}, 
\]
with the conventions $\dist(x,\emptyset)=\operatorname{diam} (\Omega)$ and $\sup \emptyset=0$. 
Classical references for this topic are  \cite{Rogers:1998th} and \cite{Falconer:1986ua}.

We start by recalling the classical Blaschke's selection Theorem (cfr. \cite{Rogers:1998th}):
\begin{thm}[Blaschke's selection principle]\label{Blaschke}
	Let $\{K_n\}$ be a sequence in $\mathcal K(\overline \Omega)$.  Then there exists a subsequence which converges in the Hausdorff metric to a set $K\in \mathcal K(\overline \Omega)$.
\end{thm}

The Hausdorff measure is not in general lower semicontinuous with respect to the convergence in the Hausdorff metric. However it is lower semicontinuous in $\mathcal{K}^{f}(\overline{\Omega})$, as stated in the following theorem (see also \cite{DalMaso:2014fp} for a more general statement). 

\begin{thm}[Gol\c ab's theorem]\label{thmGolab}
Let $\Omega$ be a bounded open set of $\R^2$ and $U$  an open subset of $\R^2$.
 Let $\{K_n\}$ be a sequence contained in $\mathcal{K}^{f}(\overline{\Omega})$
converging to a set $K$ in the Hausdorff metric. 
Then $K\in \mathcal{K}^{f}(\overline{\Omega})$ 
and 
\[
\mathcal{H}^1(K\cap U)\leq \liminf_{n\to \infty} \mathcal{H}^1(K_n\cap U).
\]
\end{thm}

The following semicontinuity result on $K^f(\Omega)$ can be found for instance in \cite{Giacomini:2002vq}.

\begin{thm}\label{thm4.1Giacomini}
Let $\varphi:\overline{\Omega}\to \R$ be a continuous function such that there exist two constants
$0\leq c_1\leq c_2$ for which
$c_1\leq \varphi(x)\leq c_2$ for every $x\in \overline{\Omega}$.
The functional
$$
K\in \mathcal{K}^f(\overline{\Omega})\to \int_{K\cap U} \varphi(x)d\mathcal{H}^1
$$
is lower semicontinuous if $\mathcal{K}^f(\overline{\Omega})$ is endowed with the Hausdorff metric.
\end{thm}

The following property of connected sets with finite length will be useful in the sequel (cfr. \cite[Proposition 2.5]{DalMaso:2014fp}).

\begin{prop}\label{connexion_by_arcs}
A connected set $C\subset \R^2$ with finite $\mathcal{H}^1$ measure is arcwise connected  and $\mathcal{H}^1(C)=\mathcal{H}^1(\overline{C})$.
\end{prop}

We conclude this section recalling that any compact arcwise connected set $E$ with finite $\mathcal{H}^1$ measure consists of a countable union of rectifiable curves, together with a set of $\mathcal{H}^1$ measure zero (see for example \cite{Falconer:1986ua}). By {\it rectifiable curve} we mean the image of a continuous injection $\psi: [a,b]\to \R^2$ with finite $\mathcal{H}^1$ measure.

\subsection{Functions of Bounded Variation}	\label{subsec:SBV}

We summarize here few basic results on the theory of {\it functions of bounded variation} that will be needed in the sequel. For a complete description of the theory one can refer for instance to \cite{AmbFusPal00-000,Fed69-000,Evans:2015wg} and the references therein. 
All the results generalize naturally to maps with values in $\R^m$, $m\geq 1$.
 
 Let $\Omega\subset \R^n$ be an open bounded domain. Given $u\in L^1(\Omega)$, we will use the notation $\mathcal{S}_{u}$ to denote the {\it approximate discontinuity set} of $u$, i.e. the set of points where $u$ does not have an approximate limit and we denote by $\mathcal{J}_{u}$ the set of {\it approximate jump points} of $u$. A function $u\in L^1(\Omega)$ is said to be of {\it bounded variation} if its distributional derivative can be represented by a finite Radon measure in $\Omega$, i.e. if there exists a vector valued Radon measure $Du=(D_1u,\dots,D_n u)$ such that
 \[
 \int_\Omega u \frac{\partial \phi}{\partial x_i} \, dx = -\int_\Omega \phi \, dD_i u \;\; ,\;\;\; \; \forall \, \phi\in C_c^{\infty}(\Omega).
 \] 
 The linear space of the functions of bounded variation is commonly denoted by $BV(\Omega)$ and can be endowed with the usual norm
\[
\norm{u}_{BV(\Omega)}=\norm{u}_{L^1(\Omega)}+|Du|(\Omega)\,
\]
that makes it a Banach space. 
Let $u_n, u \in BV(\Omega)$. We say that $\{u_n\}_n$ {\it {weakly* converges}} to $u$ in $BV(\Omega)$ if $u_n \to u$ in $L^1(\Omega)$ and the measures $Du_n$ weakly* converge to the measure $Du$ in $\mathcal{M}(\Omega, \R^n)$, that is,
\[
\lim\limits_{n \to \infty}\int_{\Omega}\varphi\, dD u_n=
\int_{\Omega}\varphi \,dD u \; ,\qquad \forall\,\varphi \in C_0(\Omega)\,.
\]
The following result is often useful.
\begin{prop}
\label{prop:crit-compact}
	Let $\{u_n\}_n \subset BV(\Omega)$. Then $u_n$ weakly* converges to $u$ in $BV(\Omega)$ if and only if $\{u_n\}_n$ is bounded in $BV(\Omega)$ and converges to $u$ in $L^1(\Omega)$. Moreover for any non-negative continuous function $f: \om \to [0,+\infty[\,$ we have the following semicontinuity property:
	\[
	\int_\om f(x) d|D_j u|(x) \leq \liminf_{n \to \infty}\int_\om f(x) d|D_j u_{n}|(x)\,.
	\]
\end{prop}

We recall that for a given function $u\in BV(\Omega)$, its distributional derivative can be decomposed as $Du=D^a +D^j u+D^c u$ where $D^a u$ is the absolutely continuous part with respect to the Lebesgue measure $\mathcal L^2$ and $D^j u$ and $ D^c u$ are the jump part and the Cantor part respectively (cfr. \cite[Section 3.9]{AmbFusPal00-000}). We also denote $D^s u:= D^j u+D^c u$ the singular part and $\tilde Du:= D^a u + D^c u$ the diffuse part of $Du$. 
 
A function $u\in BV(\Omega)$ is said to be a {\it {special function of bounded variation}} and we write $u\in SBV(\Omega)$ if the Cantor part of its derivative, $D^c u$, is zero. Then the distributional derivative of a function $u\in SBV(\Omega)$ has a special structure, i.e. it is the sum of an absolutely continuous part with respect to $\mathcal L^n$ and a $(n-1)-$rectifiable measure. The space $SBV(\Omega)$ is a closed subspace of $BV(\Omega)$.

\subsection{Caccioppoli Partitions and Piecewise constant maps}  \label{subsec:piecewise}

We summarize here the definition and the main properties of partitions of a domain $\Omega$ in sets of finite perimeter, often called Caccioppoli partitions, and of the piecewise constant functions, i.e. functions that are constant in each set of a Caccioppoli partition. These concepts have been introduced and studied for instance in \cite{Congedo:1991to,Congedo:1993dv} and \cite[Section 4.4]{AmbFusPal00-000}. 

\begin{defn} \label{ReducedBoundary} Let $E$ be a Lebesgue measurable subset of $\R^n$ and $\mathcal O$ the largest open set such that $E$ is locally of finite perimeter in $\mathcal O$. 
The {\it reduced boundary}  $\partial^r E$ is the collection of all points $x\in \operatorname{supp}|D\chi_E|\cap \mathcal O$ such that there exists in $\R^n$ the limit 
\[
\nu_E(x) := \lim_{\rho \downarrow 0} \frac{D\chi_E(B(x,\rho))}{|D\chi_E|(B(x,\rho))}
\] 
and satisfies $| \nu_E(x)|=1$. The function $\nu_E \,:\, \partial^rE \to \boldmath S^{n-1}$ is called the {\it generalized inner normal} to $E$.
\end{defn}

The {\it upper and lower densities}
of a Borel set $E\subset \R^n$  at $x$ are defined respectively by  
\[
\theta^*(E,x):= \limsup_{\rho\downarrow 0}\frac{\mathcal{L}^n(E\cap I_\rho(x))}{\omega_n \rho^n}\;\;, \;\;\;\;\theta_*(E,x):= \liminf_{\rho\downarrow 0}\frac{\mathcal{L}^n(E\cap I_\rho(x))}{\omega_n \rho^n}\,.
\]
If they agree, their common value $\theta(E,x)$ defines the {\it density} of $E$ in $x$.
{ For every $t\in [0,1]$ and every $\mathcal L^{n}$-measurable set $E\subset \R^n$ we denote by $E^t$ the set of all points where $E$ has density $t$. We use the notation $\partial^*E$ to denote the {\it essential boundary} of $E$, i.e. the set $\R^n\setminus (E^0\cup E^1)$ of points where the density is neither $0$ nor $1$. }

The structure theorem of De Giorgi (cfr. for instance Theorem 3.59 in \cite{AmbFusPal00-000}) ensures that for a measurable set $E\subset \R^n$, the reduced boundary $\partial^r E$ is countably $(n-1)$-rectifiable, $|D\chi_E|= \Haus^{n-1}\lfloor_{\partial^r E}$ and for any $x_0\in \partial^r E$ the sets $(E-x_0)/\rho$ locally converge in measure in $\R^n$ as $\rho \downarrow 0$ to the half space $H$ orthogonal to $\nu_E(x_0)$ containing $\nu_E(x_0)$. Moreover by a result of Federer (cfr. Theorem 3.61 in \cite{AmbFusPal00-000}) 
we can say that, if $E$ is a set of finite perimeter, then every point $x_0\in \partial^r E$ has density $1/2$ with respect to $E$  and $\Haus^{n-1}$-a.e. point of the essential boundary of $E$ belongs to the reduced boundary of $E$.

We can now introduce the concept of Caccioppoli partition of a set $\Omega$.

\begin{defn}
Let $\Omega\subset \R^n$ be an open set and let $\mathcal E$ be a finite or countable family of measurable sets of $\mathbb{R}^{n}$. $\mathcal E$ is said to be a  {\it Caccioppoli partition} of $\Omega$, if and only if there exists a sequence $E_{i\in \N}$ such that
\[
\mathcal E = \{ E_{i} \;:\; i \in \N\} \;,\;\; \left| \Omega \setminus \bigcup_{i=1}^{\infty} E_{i} \right| =0 \;,\;\; E_{i}=E_{i}^{1} \; \;\forall i\in \N\;, 
\]
\[
E_{i}\cap E_{j} =\emptyset \;\; \forall i\not= j\;,\;\;\sum_{i=1}^{\infty} P(E_{i},\Omega) < \infty.
\]
\end{defn}

The Federer's result on the reduced boundary of sets of finite perimeter recalled above allows us to describe the local structure of Caccioppoli partitions. Indeed, up to a $\Haus^{n-1}$-negligible set, any point of $\Omega$ either belongs to one and only one set $(E_i)^1$ or belongs to the intersection of two and only two boundaries $\partial^r E_i \cup \partial^r E_j$. The previous sentence is made precise by the following theorem.

\begin{thm}
	\label{thm:caccio-structure}
	Let $\{E_i\}_{i\in I}$ be a Caccioppoli partition of $\Omega$. Then $\Haus^{n-1}$-a.e. point of $\Omega$ is contained in 
	\[
	\bigcup_{i\in I} (E_i)^1 \, \cup \bigcup_{i,j\in I, \, i\not=j}(\partial^r E_i \cap \partial^r E_j).
	\]
\end{thm}

\begin{defn}\label{defn_pw_constant}
Let $v: \Omega\to \R^m$. We say that $v$ is piecewise constant if there exists a Caccioppoli partition of $\Omega$ and $t_i \in \R^m$ such that
\begin{equation}
\label{eq-piecewise-def}
v=\sum_i t_i\chi_{E_i}\,.
\end{equation}
\end{defn}

We recall that if $v$ is a bounded piecewise constant function on $\Omega$ represented by \eqref{eq-piecewise-def} the approximate discontinuity set, $\mathcal S_{v}$, and the approximate jump set $\mathcal J_{v}$, can be described in terms of the sets $E_{i}$. Indeed  $\Omega\setminus \mathcal S_{v}$ coincides, up to a $\Haus ^{n-1}$-negligible set, with
\[
\bigcup_{i}\big(\Omega \cap (E_{i})^{1}\big ) \cup \bigcup_{i\not= j , t_{i}=t_{j}} \big(\Omega \cap \partial^r E_i \cap \partial^r E_j \big).
\]
$J_{v}$ contains 
\[
\bigcup_{i, j : t_{i}\not= t_{j}} \big(\Omega \cap \partial^r E_i \cap \partial^r E_j\big)
\] 
and it is contained in this set up to a $\Haus ^{n-1}$-negligible set. 

Bounded piecewise constant functions can also be characterized by properties of their distributional derivatives. Indeed, if $v \in [L^{\infty}]^{m}$, them $v$ is equivalent to a piecewise constant function if and only if $v\in [BV_{loc}(\Omega)]^{m}$, $Dv$ is concentrated on $S_{v}$ and $\Haus^{n-1}(S_{u}) < \infty$. Moreover, denoting by $E_{i}$ the level sets of $v$ we have
\[
2 \Haus^{n-1}(S_{u}) =  \sum_{i} P(E_{i},\Omega).
\]

In the sequel the following compactness result for piecewise constant functions will be useful (cfr. \cite[Theorem 4.8 and Theorem 4.25]{AmbFusPal00-000} ).
\begin{thm}\label{compattezzapw} 
Let $\Omega$ be an open bounded Lipschitz set. Let $\{v_n\} \subset SBV(\Omega)$  be a sequence of piecewise constant functions such that
$\norm{v_n}_\infty + \mathcal{H}^{N-1}(\mathcal{S}_{v_n})$ is bounded.
Then, up to a subsequence, $v_n$ converges  weakly*  in $BV(\Omega)$ and 
in measure to a piecewise constant function $v$.
\end{thm}

\subsection{The vectorial pyramid in a square}\label{subsec:DMP-construction}
In the sequel we will often refer to the explicit solution to (\ref{OPGG}) constructed in \cite{Dacorogna:2008ih} for the set $E$ given by the eight matrices
\begin{equation}\label{eq:8matrici}
\begin{array}{ll}
\pm A_1=\pm\matrice{1}{0}{0}{1}, & \pm A_2=\pm\matrice{0}{1}{1}{0}, \vspace{0.2cm}
 \\ 
\pm A_3=\pm\matrice{-1}{0}{0}{1},&\pm A_4=\pm\matrice{0}{-1}{1}{0}\,.
\end{array}
\end{equation}
We rapidly review its definition. We will refer to it as the {\it vectorial pyramid}  and we will denote it with $p_v(x,y)=(p_v^1(x,y),p_v^2(x,y))$. We set $\Omega\subset \R^2$ to be the square $Q_2(0,0):=(-2,2)\times (-2,2)$. Since the two components of $p_v$ are symmetric with respect to the axes and to the lines $y=\pm x$, it is sufficient to define it only in the triangle 
\[
T=\big\{(x,y)\in \R^2: 0\leq y\leq x\leq 2\big\}.
\] 
Let 
\[
a(x,y)=\min\{1\pm x, 1\pm y\}\,,\;\;\;\;\;
b(x,y)=\max\{1-|x|, 1-|y|\}
\]
\[
c(x,y)=\left\{
\begin{array}{ll}
1-|x|, & if \,\, |x|\leq y
\\
1-y, & if \,\,|y|\leq -x
\\
1-x, & if \,\,|x|\leq -y
\\
1-|y|, & if\,\, |y|\leq x
\end{array}
\right.\;,\,\,\,\,
d(x,y)=\left\{
\begin{array}{ll}
1-x, & if \,\,|x|\leq y
\\
1+y, & if \,\,|y|\leq -x
\\
1-|x|, & if \,\,|x|\leq -y
\\
1-|y|, & if \,\,|y|\leq x
\end{array}
\right.
\]
and consider the rescaled functions $a_k(x,y)=2^{-k} a(2^k x,2^k y)$ and similarly define the functions $b_k, c_k$ and $d_k$.
Let $x_0=0$, $x_k=2-\frac{1}{2^{k-1}}$, $y_0=0$ and $y^{i}_k=\frac{i}{2^{k-1}}$ for $i=0, 1, \dots 2^{k}-1$. Define $Q_{k,i}$ as the squares  $(x_{k-1},x_k)\times (y^{i}_{k},y^{i+1}_k)$ for $i=0, 1, \dots 2^{k}-2$ (cfr. Figure 1).

\newcommand{\DMP}{
    \begin{tikzpicture}[x=5mm,y=5mm]
    \draw[->] (-5,-4.5)node[below right]{}--(4,-4.5)node[right]{$x$};\draw[->] (-4.5,-5)--(-4.5,4)node[left]{$y$};
        \draw(-4.5,-4.5)--(4,4);
        \draw(-0.5,-0.5)--(1.5,-0.5);
        \draw(-0.5,-4.5)--(-0.5,-0.5);
        \draw(-0.5,-2.5)--(1.5,-2.5);
        \draw(1.5,-4.5)--(1.5,1.5);
        \draw(2.5,-4.5)--(2.5,2.5);
        \draw(1.5,-1.5)--(2.5,-1.5);
        \draw(1.5,-2.5)--(2.5,-2.5);
\draw(1.5,-3.5)--(2.5,-3.5);
\draw(1.5,-0.5)--(2.5,-0.5);
\draw(1.5,0.5)--(2.5,0.5);
\draw(1.5,1.5)--(2.5,1.5);
\draw(2.5,1.5)--(3,1.5);
\draw(2.5,2)--(3,2);
\draw(2.5,1)--(3,1);
\draw(2.5,0.5)--(3,0.5);
\draw(2.5,0)--(3,0);
\draw(2.5,-0.5)--(3,-0.5);
\draw(2.5,-1)--(3,-1);
\draw(2.5,-1.5)--(3,-1.5);
\draw(2.5,-2)--(3,-2);
\draw(2.5,-2.5)--(3,-2.5);
\draw(2.5,-3)--(3,-3);
\draw(2.5,-3.5)--(3,-3.5);
\draw(2.5,-4)--(3,-4);
\draw(3,-4.5)--(3,3);
\draw (-0.5,-3.5)node[left]{\tiny{$Q_{1,0}\cap T$}};
\draw (1.5,-3.5)node[left]{\tiny{$Q_{2,0}$}};
\draw (1.5,-1.5)node[left]{\tiny{$Q_{2,1}$}};
\draw (1.7,0.1)node[left]{\tiny{$Q_{2,2}\!\cap\! T$}};
\draw (-3.5,-4.9)node[left]{\footnotesize{$x_0$}};
\draw (0.3,-4.9)node[left]{\footnotesize{$x_1$}};
\draw (2.3,-4.9)node[left]{\footnotesize{$x_2$}};
\draw (3.2,-4.9)node[right]{\footnotesize{$x_{\infty}\!\!=\!\!2$}};
\draw(3.5,-4.6)--(3.5,-4.4);                
                \end{tikzpicture}
}
\begin{figure}[h]
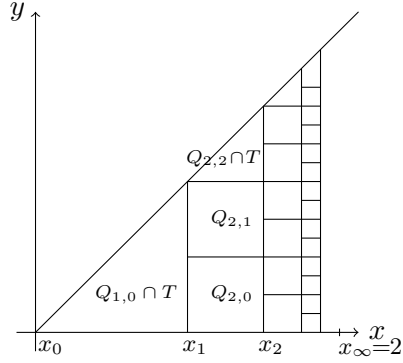

\label{fig:DMP-sol}
    \centering
    \DMP
    \caption{The distribution of the squares $Q_{k,i}$ in $T$}
    \end{figure}

The first component $p_v^1$ of the map $p_v$ is  defined as
\[
p_v^1(x,y)=a_k\left(
x-\frac{x_{k-1}+x_k}{2},
y-\frac{y^i_{k}+y^{i+1}_k}{2}
\right)
\]
for $(x,y)\in Q_{k,i}\cap T$. The second component $p_v^2$ is given by
\[
p_v^2(x,y)=
\left\{
\begin{array}{ll}
d_k\left(
x-\frac{x_{k-1}+x_k}{2},
y-\frac{y^i_{k}+y^{i+1}_k}{2}
\right), \textnormal{if} \,\, i\,\, \textnormal{is even and}\,\, i\in \{0,\dots 2^k-4\} \vspace{0.1cm}
\\
c_k\left(
x-\frac{x_{k-1}+x_k}{2},
y-\frac{y^i_{k}+y^{i+1}_k}{2}
\right), \textnormal {if}\,\, i\,\, \textnormal{is odd and}\,\, i\in \{1,\dots 2^k-3\} \vspace{0.1cm}
\\
b_k\left(
x-\frac{x_{k-1}+x_k}{2},
y-\frac{y^i_{k}+y^{i+1}_k}{2}
\right), \textnormal{if} \,\, i=2^k-2\,.
\end{array}
\right.
\]
The map $p_v$ belongs to  $W^{1,\infty}_0(\Omega;\R^2)$ and is a solution of the Dirichlet problem \eqref{OPGG} (cfr. \cite[Theorem 1]{Dacorogna:2008ih}). Moreover it is worth to observe  that $p_v$ attains the homogeneous boundary datum  in a fractal way. To be more precise we observe that on any square $Q_{k,i}$ the gradient $Dp_v$ of $p_v$ is discontinuous on the boundary of $Q_{k,i}$, on the diagonals and on the segments parallel to the axes passing from the center of $Q_{k,i}$. From this observation it is not difficult to realize that for any measurable set $B \subset \subset Q_2(0,0)$, defined the eight sets 
\[ 
\Omega^{p_v}_{\pm i}:= \big\{(x,y) \in Q_2(0,0)\; : \; D{p_v}(x,y) = \pm A_i \big\}\,,
\]
the family $B_{\pm i}:={\Omega^{p_v}_{\pm i}} \cap B$ for $i\in \{1,2,3,4\}$ is a Caccioppoli partition for $B$. If instead $B$ is an open set such that $B\cap \partial Q_2(0,0)\not= \emptyset$, then the $\Haus^1$-measure of the intersection between the set where $Dp_v$ is discontinuous and  $B$ is infinite.
 
\begin{rem}
	\label{rem:stripdistribution} Given $k\in \N$ with $k>1$, consider an even index $j$ less than $2^k-2$. 
	By using the values of $D p_v$ in $Q_{k,j}$ and $Q_{k,j+1}$ as represented in figures \ref{primo_quadrato} and \ref{secondo_quadrato},
it is easy to prove the following estimate for any $i\in \{1,2,3,4\}$
	\[
	\mathcal{L}^2 \big((Q_{k,j} \cup Q_{k,j+1}) \cap \Omega_{\pm i}^{p_v}\big) \geq \frac{1}{8} \left(\frac{1}{2^k}\right)^2.
	\]


\newcommand{\mapad}{
	\begin{tikzpicture}[x=5mm,y=5mm]
	\draw(-4,4)--(4,4);
	\draw(-4,-4)--(4,-4);
	\draw(4,-4)--(4,4);
	\draw(-4,-4)--(-4,4);
	\draw(-4,0)--(4,0);
	\draw(0,-4)--(0,4);
	\draw(-4,-4)--(4,4);
	\draw(-4,4)--(4,-4);
\draw (1.4,1.5)node[right]{{$-A_1$}};
\draw (2,-1.5)node[right]{{$A_3$}};
\draw (0.8,-3)node[right]{{$-A_4$}};		
\draw (-2.2,-3)node[right]{{$A_2$}};		
\draw (-3.8,-1.5)node[right]{{$A_1$}};	
\draw (-3.8,1.5)node[right]{{$A_1$}};	
\draw (-2.5,2.5)node[right]{{$-A_2$}};	
\draw (0.3,2.5)node[right]{{$-A_2$}};	
		
		\end{tikzpicture}
}

\begin{figure}[h]
    \centering
    \mapad
    \caption{Values of $Dp_v$ on $Q_{k,j}$}
    \label{primo_quadrato}
    \end{figure}


\newcommand{\mapac}{
	\begin{tikzpicture}[x=5mm,y=5mm]
	\draw(-4,4)--(4,4);
	\draw(-4,-4)--(4,-4);
	\draw(4,-4)--(4,4);
	\draw(-4,-4)--(-4,4);
	\draw(-4,0)--(4,0);
	\draw(0,-4)--(0,4);
	\draw(-4,-4)--(4,4);
	\draw(-4,4)--(4,-4);
\draw (1.4,1.5)node[right]{{$-A_1$}};
\draw (2,-1.5)node[right]{{$A_3$}};
\draw (0.8,-3)node[right]{{$-A_4$}};		
\draw (-2.2,-3)node[right]{{$-A_4$}};		
\draw (-3.8,-1.5)node[right]{{$-A_3$}};	
\draw (-3.8,1.5)node[right]{{$-A_3$}};	
\draw (-1.5,2.5)node[right]{{$A_4$}};	
\draw (0.3,2.5)node[right]{{$-A_2$}};	
		
		\end{tikzpicture}
}

\begin{figure}[h]
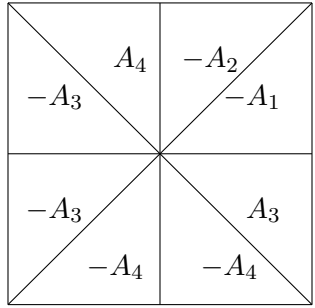

    \centering
    \mapac
    \caption{Values of $Dp_v$ on $Q_{k,j+1}$}
    \label{secondo_quadrato}
    \end{figure}
	
\end{rem}
 
We end this paragraph with a lemma that we will use in the last section, to prove that the variational problem that we will consider is well posed.

\begin{lem}\label{lemma:H3pv}
There exist a constant $c>0$ such that for any $(x,y)\in \partial Q_2(0,0)$ and any $r< \frac{1}{4}$ we have
\begin{equation}\label{eq:H3SV}
\mathcal{L}^2 \big(B((x,y),r)\cap \Omega^{p_v}_{\pm i} \big) \geq c r^2,
\end{equation}
for any $i\in \{ 1, 2,  3,  4\}$.
\end{lem}
\begin{proof}
By the symmetries of $p_v$, it is sufficient to  consider a point $(x,y)$ in the triangle $T$ and lying on $\partial Q_2(0,0)$, i.e. we can suppose $(x,y)=(2,y^0)$ with $0\leq y^0\leq 2$. Given $r>0$ let $B_{r}$ be the open ball in the $l_{\infty}$ norm of $\R^2$ centered at $(x,y)$ with radius $r$, i.e.
\[
B_r:=\big\{(x,y)\in \R^2 \;:\; \max\{ |x-2|, |y-y^0|\}< r \big \}.
\] 

We observe that, since $x_k=2-\frac{1}{2^{k-1}}$, for any given $r \in (0,\frac{1}{4})$, choosing $k=-[\log_2(r)]+1$, one has
\[
x_{k}= 2-\frac{1}{2^{k-1}} \leq 2-r\leq x_{k+1} = 2-\frac{1}{2^{k}}
\] 
and  
\[
x_{k+1}= 2-\frac{1}{2^{k}}\leq 2-\frac r2\leq x_{k+2}=2-\frac{1}{2^{k+1}}.
\]
Therefore $B_r$ contains at least two adjacent squares of type $Q_{k,i}$, $Q_{k,i+1}$. By Remark \ref{rem:stripdistribution} and by the estimate
$r\leq \frac{1}{2^{k-1}}$,
we have
that 
\[
\mathcal{L}^2 \big(B((x,y),r)\cap \Omega^{p_v}_{\pm i} \big) \geq  \mathcal{L}^2 (B_{\frac{r}{2}} \cap \Omega_{\pm i}^{p_v}) 
\geq \frac{1}{8} \left(\frac{1}{2^{k+1}}\right)^2
\geq \frac{r^2}{128}\,,
\]
that is, \eqref{eq:H3SV} holds true for    $c=\frac{1}{128}$.
\end{proof}

\begin{rem}
	\label{rem:vitali+pyramid}
The map $p_v$ can be used to get a solution to (\ref{OPGG}) in any
open and bounded
 $\Omega \subset \R^2$. Indeed, let
$\{Q_i\}_{i\in I}$ be  a family of disjoint squares with sides parallel to the axes
that covers $\Omega$ up to a set of zero Lebesgue measure.
We can construct a solution $u  \in W^{1,\infty}(\Omega)$
to \eqref{OPGG}
by defining it as a rescaled vectorial pyramid in any $Q_i$.
In the sequel we will often use this construction.
\end{rem}

\section{Fine properties of solutions}
\label{sec:PropertiesSolutions}

In this section we are going to recall some geometric and topological properties 
of Lipschitz vector valued maps whose gradient takes only a finite number of values. Similar problems have been studied in the scalar setting of Lipschitz functions in \cite{Bellettini:2006wy}. 

Let $\Omega\subset \mathbb R^n$ be open and bounded, $u :\Omega \to \mathbb R^N$ be in $W^{1,\infty}(\Omega; \mathbb R^N)$. Let 
$E\subset \mathbb R^{n\times N}$ be a set composed by a finite number of matrices:
\[
E = \big\{ A_1,\dots, A_k\big\},\, k\in \N\,.
\]
We assume in the sequel that $u$ solves the inclusion
\begin{equation}
	\label{eq:inclusion}
	Du \in E \; \;\textrm{for a.e. } x\in \Omega.
	\end{equation}
As it has been done in \cite{Bellettini:2006wy}, we define the {\it singular set} of $u$ as  
\[
\Sigma^u_E:= \big\{ x \in \Omega \;: \; \textrm{either } u \textrm{ is not differentiable at } x \textrm{ or } Du(x)\not\in E\big\},
\]
and the {\it regular set } of $u$ as $\Omega \setminus \Sigma_E^u$.  

For a solution  $u$ to (\ref{eq:inclusion}), let us consider for $i=1,\dots,k$ the sets
\[ 
\Omega^u_i:= \left\{x \in \Omega\; : \; Du(x) = A_i \right\}.
\]
We note that the measurable set $\Omega_i^u$ is not necessarily of finite perimeter. 
With the aim of distinguishing somehow the {\it bad} points of $\Sigma^u_E$ from the {\it good} ones, we define $\Gamma_u$ and $\Sigma_\infty^u$ as follows:
\begin{defn}\label{def:sigmainfty}
Let $u$ be a solution to (\ref{eq:inclusion}).
A point $x_0\in \Omega$ belongs to $\Gamma_u$ if there exists a ball $B(x_0,r)\subset \Omega$ centered at  $x_0$ such that the sets 
$\Omega^u_i\cap B(x_0,r)$ 
for $i\in \{1,\dots,k\}$ form a Caccioppoli partition of $B(x_0,r)$. We set 
\[
\Sigma_\infty^u:=\partial \Omega \cup ({\Omega\setminus \Gamma_u}).
\]
\end{defn}

The set $\Sigma_\infty^u$ will play a central role in our analysis and can be seen roughly speaking as the set of points where the singularities of $Du$ accumulate and have a fractal behavior. We now state few basic properties of $\Sigma^u_\infty$.

\begin{lem}
\label{lem:closureSigma}
Let $u$ be a solution to (\ref{eq:inclusion}). Then $\Sigma^u_{\infty}$ is closed.
\end{lem}
\begin{proof}
Being $[\Sigma^{u}_{\infty}]^{c}=[\Omega]^{c}\cup \Gamma_{u}$, the claim will follow once we have proved that $\Gamma_u$ is open. If $x_0$ belongs to $ \Gamma_u$ then there exists $r_0>0$
such that $\Omega^u_i \cap B(x_0,r_0)$ for $i\in \{1,\dots,k\}$ is a Caccioppoli partition of $B(x_0,r_0)$. Then it follows that  $B(x_0,r_0/2)$ is contained in $\Gamma_u$. This proves the lemma.  
\end{proof}

For any $x \in \Gamma_u$ we denote by $\rho(x)>0$ the radius of the largest ball centered in $x$ where $Du$ is a piecewise constant in the sense of Definition \ref{defn_pw_constant}, i.e.
\[
\rho(x):= \sup \big\{ \rho\;:\; \Omega^u_i \cap B(x,\rho) \textrm{ form a Caccioppoli partition of } B(x,\rho)\big\}.
\]
It is not hard to realize that $\rho(x)= \dist_{\mathcal{H}}(x,\Sigma_\infty^u)$, as it is proven in the next lemma. This implies that $Du$ is a map of bounded variation far away from $\Sigma_\infty^u$	.

\begin{lem}
	\label{lem:SBV-farfromSigma}
	Let $u$ be a solution to (\ref{eq:inclusion}).
	Then $\rho(x)= \dist_{\mathcal{H}}(x,\Sigma_\infty^u)$, for any $x \in \Gamma_u$. Moreover, for any $K \Subset \Gamma_u$, $Du$  is piecewise constant in $K$ and therefore $Du\in SBV(K)$ .
\end{lem}  
\begin{proof}
	For the first claim it is enough to show that there exists $y\in \partial B(x, \rho(x)) \cap \Sigma_\infty^u$. Arguing  by contradiction, we assume that 
	\begin{equation}\label{eq:inclusionecontraddizione}
	\partial B(x,\rho(x))\subset \Gamma_u\,.
	\end{equation}	
	Then for any $y\in \partial B(x,\rho(x))$ we have $\rho(y)>0$. Let 
	\[
	\bar\rho:= \inf \big\{ \rho(y) :  y\in \partial B(x,\rho(x)) \big\} \geq 0.
	\]
	We claim that $\bar\rho=0$. Otherwise there would exist a finite collection of points $\{y_i\}_{i=1}^l \subset \partial B(x,\rho(x))$ and $\delta>0$ such that
	\[
	B(x,\rho(x)+\delta) \subset B(x,\rho(x)) \cup \bigcup_{i=1}^l B(y_i,\bar\rho).
	\] 
	Since $Du$ is piecewise constant in $B(y_i,\bar\rho)$ fon any $i\in \{1,\cdots,l\}$ as well as in $B(x,\rho(x))$, the previous inclusion would imply that $Du$ is piecewise constant in $B(x,\rho(x)+\delta)$. This would be in contradiction with the definition of $\rho(x)$. 
	\\
	Since $\bar\rho=0$ we can consider a sequence $\{y_n\}_{n\in\N}\subset \partial B(x,\rho(x))$ such that 
$\rho(y_n) \to  0$ as $n\to \infty$.
	By compactness there exists $y_0 \in \partial B(x,\rho(x))$ with $y_n\to y_0$, as $n\to \infty$. 
	We are going to prove that $y_0\notin \Gamma_u$. This will be a contradiction with
(\ref{eq:inclusionecontraddizione}) and then prove the claim. To this aim we observe that if  $y_0\in \Gamma_u$, then $\rho(y_0)>0$. Let $m\in \N$ be sufficiently large such that
	\[
	B(y_m,2\rho(y_m)) \subset B(y_0,\rho(y_0)):
	\]  
this is a contradiction with the definition of $\rho(y_m)$. 

We will now deduce that 
$Du$ 
is piecewise constant and belongs to $SBV(K)$ for any $K \subset \Gamma_u$. 
Let $\delta < \dist(K, \Sigma^u_\infty)$. We have that $Du$ is piecewise constant in $B(y, \delta/2)$ for any $y\in K$. Moreover, since $K$ is bounded, we can cover it with finitely many balls $B(x_i,\delta/2)$ for $i\in\{1,\dots,h\}$ with $x_i\in K$, $x_i\not\in \bigcup_{j\not=i}B(x_j,\delta/2)$ and the number of overlapping balls is at most 9 (cfr. for instance \cite[Sec 2.4]{AmbFusPal00-000}). Therefore 
	$Du$ is piecewise constant in the whole $K$. This proves that $Du$ is $SBV(K)$.  
\end{proof}

Our work stems from the idea that the smaller $\Sigma_\infty^u$ is, the better the solution $u$ to \eqref{eq:inclusion} is. So we try to define an energy integral that somehow measures how large is $\Sigma_\infty^u$ for a given solution to \eqref{eq:inclusion} and we study the associated minimization problem.

It is clear that $\Sigma_\infty^u$ can be as bad as we can think of, for instance it could be not even with locally finite $\mathcal{H}^1$-measure in $\Omega$, as the following simple example shows. 

\begin{ex}
Let $\Omega=(-2,2)\times (-2,2)$ and consider the set of matrices $E$  defined by \eqref{eq:set-E}. Let $f(x)=\sin\left(\frac{1}{|x|}\right)$ and define 
\[
G_f:=\big\{ (x,y)\in \Omega:\, x\not= 0 \textrm{ and } y=f(x) \big\} \cup \big\{(0,y)\in \Omega:\, \textrm{ with } y\in [-1,1] \big\}. 
\]
Since $\mathcal L^2(G_f)=0$ we can argue as in Remark \ref{rem:vitali+pyramid} and define a solution $u_f$ to \eqref{OPGG} associated to a Vitali covering of $\Omega \setminus G_f$ made up by squares with sides parallel to the axes. The function $u_f$ solves \eqref{eq:inclusion} and $\Sigma_\infty^{u_f}$ contains the portion of the graph of $f$ that lies in $\Omega$. It is easy to check that $\Sigma_\infty^{u_f}$ is not locally of finite $\mathcal{H}^1$-measure in $\Omega$ by considering a neighborhood of the origin.
\end{ex}

This example motivated us to impose some structural conditions on $\Sigma_\infty^u$ in order to restrict the class of solutions of \eqref{eq:inclusion} to the ones that do not exhibit these  pathological behaviors. For any $\delta>0$ let us define
\[
\Omega_\delta:= \big\{x \in \Omega \;:\; \dist(x ,\partial \Omega) > \delta \big\}\,.
\]  
\begin{defn}
	\label{def:S}
	Let $E\subset \mathbb R^{2\times 2}$ contain only a finite number of matrices. We denote by $\mathcal{S}$ the collection of maps $u \in W^{1,\infty}_0(\Omega,\R^2)$ such that $Du(x)\in E$ for almost every $x \in \Omega$ and $\Sigma^u_\infty$ satisfies the following two conditions: 
	\begin{enumerate}
\item [{(H1)}]
$(\Sigma^u_{\infty}\cap \Omega_{\delta})\cup \partial \Omega_\delta$ is connected for any $\delta$ such that $\Omega_\delta \not= \emptyset$;
\item [{(H2)}]
$\Sigma^u_{\infty}$ is locally of finite $\mathcal{H}^1$-measure in $\Omega$.
\end{enumerate}
\end{defn}

The connectedness property seems to be natural if we think about the solutions constructed as in Remark \ref{rem:vitali+pyramid}, but cannot be considered for granted for any solution to (\ref{OPGG}) as the following example shows. 

\begin{ex}	
\label{example:fisarmonica}
Let 
\[
A_1=\matrice{1}{0}{0}{1},\; A_2=\matrice{0}{1}{1}{0},\;A_3=\matrice{-1}{0}{0}{1},\;
A_4=\matrice{0}{-1}{1}{0}\,.
\]
We will consider the partition of the "double frame" 
$$
D_{stl}=\big\{(x_1,x_2)\in \R^2: s\leq \norm{(x_1,x_2)}_{l^\infty}\leq t\big\}\cup \big\{(x_1,x_2)\in \R^2: t\leq \norm{(x_1,x_2)}_{l^\infty}\leq l\big\}
$$ 
where $0<s<t<l$, composed by the sets $I_{\pm}^i, O_{\pm}^i$, 
for $i\in \{1, 2, 3, 4\}$, as in  Figure \ref{partition_doubleframe}. 

\newcommand{\trapeziooo}{
	\begin{tikzpicture}[x=5mm,y=5mm]
	\draw[->] (-7,0)node[below right]{}--(7.5,0)node[right]{\footnotesize{$x_1$}};\draw[->] (0,-7)--(0,7.5)node[left]{\footnotesize{$x_2$}};
	\draw(-6,6)--(6,6);
	\draw(-4.5,4.5)--(4.5,4.5);
	\draw(-3.2,3.2)--(3.2,3.2);
	\draw(-6,-6)--(6,-6);
	\draw(-4.5,-4.5)--(4.5,-4.5);
	\draw(-3.2,-3.2)--(3.2,-3.2);
	\draw(3.2,-3.2)--(3.2,3.2);
	\draw(6,-6)--(6,6);
	\draw(4.5,-4.5)--(4.5,4.5);
	\draw(-3.2,-3.2)--(3.2,-3.2);
	\draw(-6,-6)--(-6,6);
	\draw(-4.5,-4.5)--(-4.5,4.5);
	\draw(-3.2,-3.2)--(-3.2,3.2);
	\draw(3.2,3.2)--(6,6);
	\draw(3.2,-3.2)--(6,-6);
	\draw(-3.2,3.2)--(-6,6);
	\draw(-3.2,-3.2)--(-6,-6);
	\draw (4.5,1.5)node[left]{\footnotesize{$I_+^4$}};
	\draw (6,1.5)node[left]{\footnotesize{$O_-^2$}};
	\draw (2,3.75)node[left]{\footnotesize{$I_+^3$}};
	\draw (-1,3.75)node[left]{\footnotesize{$I_+^1$}};
	\draw (2,-4)node[left]{\footnotesize{$I_-^1$}};
\draw (-1,-4)node[left]{\footnotesize{$I_-^3$}};
\draw (-0.9,-5.5)node[left]{\footnotesize{$O_+^1$}};
	\draw (-3.15,1.5)node[left]{\footnotesize{$I_-^2$}};
		\draw (2.2,5)node[left]{\footnotesize{$O_-^1$}};
		\draw (-0.8,5)node[left]{\footnotesize{$O_-^3$}};
	\draw (6,-2)node[left]{\footnotesize{$O_-^4$}};
		\draw (4.5,-2)node[left]{\footnotesize{$I_+^2$}};
	\draw (2.1,-5.5)node[left]{\footnotesize{$O_+^3$}};
	\draw (-4.5,1.5)node[left]{\footnotesize{$O_+^4$}};
	\draw (-4.5,-2)node[left]{\footnotesize{$O_+^2$}};
		\draw (-3.15,-2)node[left]{\footnotesize{$I_-^4$}};
	\draw (3.85,-0.3)node[left]{\footnotesize{$s$}};
	\draw (4.35,-0.3)node[right]{\footnotesize{$t$}};
\draw (5.85,-0.3)node[right]{\footnotesize{$l$}};
		
		\end{tikzpicture}
}
\begin{figure}[h]
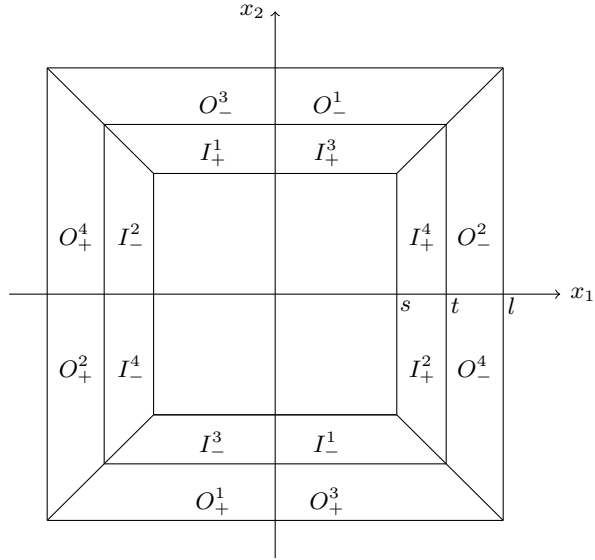

	\centering
	\trapeziooo
	\caption{The partition of  $D_{stl}$}\label{partition_doubleframe}
	\end{figure}
In $D_{stl}$ we define the family of  continuous affine piecewise maps, given by
\begin{equation}\label{familyalphabeta}
(x_1,x_2)\mapsto \begin{cases}
	 \pm A_i \cdot ( x_1, x_2)+(\alpha, \beta), &  (x_1,x_2)\in I_{\pm}^i \\
	 \pm A_i \cdot ( x_1, x_2)+(\alpha, 2t+\beta), &  (x_1,x_2)\in O_{\pm}^i
\end{cases}
\end{equation}
for $i\in \{1, 2, 3, 4\}$, where $\alpha, \beta \in \R$.
The gradient of these maps in $D_{stl}$ is constant on each of the sets $I_{\pm}^i, O_{\pm}^i$, 
for $i\in \{1, 2, 3, 4\}$, as expressed in  Figure \ref{gradient_doubleframe}.
\newcommand{\trapeziograd}{
	\begin{tikzpicture}[x=5mm,y=5mm]
	\draw[->] (-7,0)node[below right]{}--(7.5,0)node[right]{\footnotesize{$x_1$}};\draw[->] (0,-7)--(0,7.5)node[left]{\footnotesize{$x_2$}};
	\draw(-6,6)--(6,6);
	\draw(-4.5,4.5)--(4.5,4.5);
	\draw(-3.2,3.2)--(3.2,3.2);
	\draw(-6,-6)--(6,-6);
	\draw(-4.5,-4.5)--(4.5,-4.5);
	\draw(-3.2,-3.2)--(3.2,-3.2);
	\draw(3.2,-3.2)--(3.2,3.2);
	\draw(6,-6)--(6,6);
	\draw(4.5,-4.5)--(4.5,4.5);
	\draw(-3.2,-3.2)--(3.2,-3.2);
	\draw(-6,-6)--(-6,6);
	\draw(-4.5,-4.5)--(-4.5,4.5);
	\draw(-3.2,-3.2)--(-3.2,3.2);
	\draw(3.2,3.2)--(6,6);
	\draw(3.2,-3.2)--(6,-6);
	\draw(-3.2,3.2)--(-6,6);
	\draw(-3.2,-3.2)--(-6,-6);
	\draw (4.5,1.5)node[left]{\footnotesize{$A_4$}};
	\draw (6,1.5)node[left]{\footnotesize{$-A_2$}};
	\draw (2,3.75)node[left]{\footnotesize{$A_3$}};
	\draw (-1,3.75)node[left]{\footnotesize{$A_1$}};
	\draw (2,-4)node[left]{\footnotesize{$-A_1$}};
\draw (-1,-4)node[left]{\footnotesize{$-A_3$}};
\draw (-0.9,-5.5)node[left]{\footnotesize{$A_1$}};
	\draw (-3.15,1.5)node[left]{\footnotesize{$-A_2$}};
		\draw (2.2,5)node[left]{\footnotesize{$-A_1$}};
		\draw (-0.8,5)node[left]{\footnotesize{$-A_3$}};
	\draw (6,-2)node[left]{\footnotesize{$-A_4$}};
		\draw (4.5,-2)node[left]{\footnotesize{$A_2$}};
	\draw (2.1,-5.5)node[left]{\footnotesize{$A_3$}};
	\draw (-4.5,1.5)node[left]{\footnotesize{$A_4$}};
	\draw (-4.5,-2)node[left]{\footnotesize{$A_2$}};
		\draw (-3.15,-2)node[left]{\footnotesize{$-A_4$}};
	\draw (3.85,-0.3)node[left]{\footnotesize{$s$}};
	\draw (4.35,-0.3)node[right]{\footnotesize{$t$}};
\draw (5.85,-0.3)node[right]{\footnotesize{$l$}};
		
		\end{tikzpicture}
}
\begin{figure}[h]
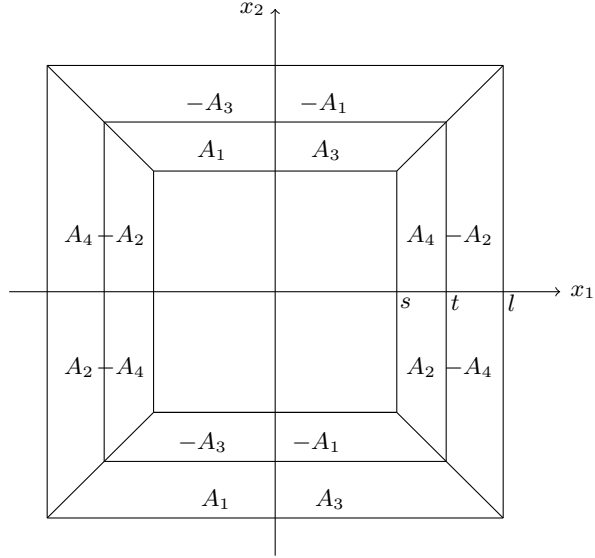

	\centering
	\trapeziograd
	\caption{The gradient in the double frame $D_{stl}$}
	\label{gradient_doubleframe}
	\end{figure}

We remark that a map of the above family is uniquely determined by its value in one point of the boundary, say the point $(l,0)$, which 
equals $(\alpha, -l+2t+\beta)$.
Moreover, for $q<r<s$, we can define in the frame  $D_{qrs}$ a similar map as above, in such a way that
we have a continuous affine piecewise map in 
$D_{qrs}\cup D_{stl}$. Indeed, 
 the directional derivatives along the overlapping boundaries agree, since  $A_4\cdot (0,1)$ and $-A_2\cdot (0,1)$ are equal
as well as $A_3\cdot (1,0)$ and $-A_1 \cdot (1,0)$. This guarantees the continuity
of the map in 
$D_{qrs}\cup D_{stl}$.

Let $\{s_j\}_{j=1}^{\infty}$ be a decreasing sequence of real positive numbers, converging to 0 and such that $s_1=1$. The square
$[-1,1]\times [-1,1]$ can be covered
by  the disjoint (up to their boundaries) "double frames"
$D_{s_{j+2}s_{j+1}s_{j}}$ (cfr. Figure \ref{doubleframesj}). 

\newcommand{\doubleframes}{
	\begin{tikzpicture}[x=5mm,y=5mm]
	\draw[->] (-6,0)node[below right]{}--(6,0)node[right]{\Small $x_1$};\draw[->] (0,-6)--(0,6)node[left]{\Small $x_2$};
	\draw(-5,5)--(5,5);
	\draw(-4,4)--(4,4);
	\draw(-2.5,2.5)--(2.5,2.5);
	\draw(-3.5,3.5)--(3.5,3.5);
		\draw(-2.5,2.5)--(2.5,2.5);
	\draw(-5,-5)--(5,-5);
	\draw(-4,-4)--(4,-4);
	\draw(-3.5,-3.5)--(3.5,-3.5);
	\draw(3.5,-3.5)--(3.5,3.5);
		\draw(2.5,-2.5)--(2.5,2.5);
	\draw(5,-5)--(5,5);
	\draw(4,-4)--(4,4);
	\draw(-3.5,-3.5)--(3.5,-3.5);
	\draw(-5,-5)--(-5,5);
\draw(-1,-1)--(-1,1);
\draw(-1,-1)--(1,-1);
\draw(1,-1)--(1,1);
\draw(-1,1)--(1,1);	
\draw(-2.5,-2.5)--(-2.5,2.5);
	\draw(-4,-4)--(-4,4);
	\draw(-3.5,-3.5)--(-3.5,3.5);
	\draw(1,1)--(5,5);
	\draw(1,-1)--(5,-5);
	\draw(-1,1)--(-5,5);
	\draw(-1,-1)--(-5,-5);
	\draw(-2.5,-2.5)--(2.5,-2.5);
\draw (4.8,-0.2)node[right]{\tiny{$s_1$}};
\draw (0.8,-0.2)node[right]{\tiny{$s_5$}};	
\draw (-0.1,-0.2)node[right]{\tiny{$s_{\dots}$}};		
\draw (3.8,-0.2)node[right]{\tiny{$s_2$}};		
\draw (3.2,-0.2)node[right]{\tiny{$s_3$}};
\draw (2.3,-0.2)node[right]{\tiny{$s_4$}};				
		\end{tikzpicture}
}
\begin{figure}[h]
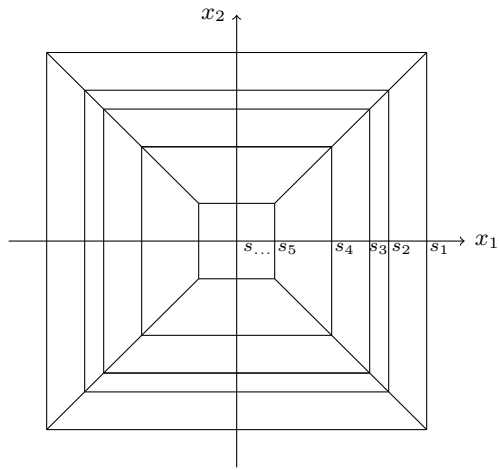

	\centering
	\doubleframes
	\caption{$\{(x_1,x_2): \norm{(x_1,x_2)}_{l^{\infty}}\leq 1\}$}\label{doubleframesj}
	\end{figure}

Now, we choose $\alpha, \beta$ in (\ref{familyalphabeta})
such that
the value of the map in $(1,0)$ is equal to $(0,1)$. 
Defining  $u$ as before on each "double frame", one can construct
 a continuous piecewise affine far from the origin map (having a fractal behavior near 0) with $Du\in E$ for almost every $x\in (-1,1)\times (-1,1)$. 
The jumps of its gradient are supported on a set containing  the boundary of the frames, whose length is 
 $$
 8 \sum_{n=2}^{\infty}s_{n+1}
 $$
and choosing for example $s_j=\frac 1j$, one has that the last quantity is infinite.

Nevertheless, it can be easily verified that this map belongs to $W^{1,\infty}(\Omega)$. To this aim, we analyze its behavior on the segment $(x_1,0)$, $0\leq x_1\leq 1$. Since the map is affine on any double frame of the construction, it is sufficient to compute the  sequence of values in $(s_{2n},0)$ and $(s_{2n+1},0)$: one has
\begin{equation}\label{eq:valuesfisarmonica}
\left(0, -s_{2n}+2\sum_{k=1}^{2n-1}(-1)^{k+1}s_k\right)\,,\,\,\,\,\,\,\,\,\,\,\,
\left(0,s_{2n+1}+2\sum_{k=1}^{2n}(-1)^{k+1}s_k\right)
\end{equation}
respectively. At the limit as $n\to \infty$,  $u$ is finite, since $s_n$ converges to $0$ and is decreasing.

Now, let 
$$
v(x_1,x_2)=\frac 12u(2x_1-1,2x_2-1)\,, \,\,\,\,\,\,\,\,\,(x_1,x_2): 0\leq x_1\leq x_2\leq 1
$$
and extend $v$ to $[-1,1]\times[-1,1]$ by symmetries with respect to 
$x_2=x_1, x_1=0, x_2=0$.
Observe that this map has the same value as the map defined for the construction of $p_v$ in subsection \ref{subsec:DMP-construction} on the boundary of $[-1,1]\times[-1,1]$. Indeed on the segment $(1,x_2)$, for $0\leq x_1\leq 1$, these maps equal $(0,1-|2x_2  -1|)$. Consequently, defined $\Omega=(-2,2)\times (-2,2)$, we can extend  $v$ in $\Omega \setminus [-1,1]\times[-1,1]$ considering the map defined in subsection \ref{subsec:DMP-construction}. In this way $Du\in E$, $u=0$ on $\partial \Omega$ and $\Sigma_{\infty}^u$ contains four isolated points, i.e. $\left(\frac 12,\pm\frac 12\right), \left(-\frac 12,\pm\frac 12\right)$.
\end{ex}

\begin{rem}
\label{rem:segmentofisarmonica}
Observe that the previous construction  cannot be adapted to define a sequence of "double frames" which, roughly speaking, does not converge to a point.
More precisely, let $s_j$ be a decreasing sequence converging to $s_{\infty}>0$, such that $s_1=1$.
The set 
\[
\big\{(x_1,x_2)\in \R^2: s_{\infty}\leq \norm{(x_1,x_2)}_{l^{\infty}}\leq 1\big\}
\]
can be covered by  disjoint (up to their boundaries) "double frames". If one defines a map $u$ as before, one has a solution to $Du\in E$ which is not bounded. Indeed, the sequence of values (\ref{eq:valuesfisarmonica}) does not converge, as $n\to \infty$. We explicitly observe that at any regular point of the boundary of the square with side $s_\infty$ the fractal behavior of $Du$ is determined by the alternation of just two values. In the Example \ref{example:fisarmonica} instead the fractalization is due to the accumulation of at least three values of $Du$. 
\end{rem}

\section{Exploiting the energy bound}
\label{sec:energybound}

As already explained in the introduction, we will use a variational approach to  select a solution to (\ref{OPGG}). Our aim is to isolate 
a solution $u$ with the smallest possible set of irregularities, that is, the smallest $\Sigma^u_{\infty}$ (see Definition \ref{def:sigmainfty}).
To do that, we define  
the following functional:
\[
\mathcal{F}(u)=\int_{\Omega}\dist(x,\partial \Omega) \chi_{\Sigma^u_{\infty}}d\mathcal{H}^1+\sum_{i,j=1}^2\int_{\Omega}\big(\dist(x,\Sigma^u_{\infty})\big)^{\alpha} \,d|Du^j_{x_i}|
\]
over the set $\mathcal{S}$ introduced in Definition \ref{def:S}.
We recall the assumptions on
$\Sigma^u_\infty$: 
	\begin{enumerate}
\item [{(H1)}]
$(\Sigma^u_{\infty}\cap \Omega_{\delta})\cup \partial \Omega_\delta$ is connected for any $\delta$ such that $\Omega_\delta \not= \emptyset$;
\item [{(H2)}]
$\Sigma^u_{\infty}$ is locally of finite $\mathcal{H}^1$-measure in $\Omega$.
\end{enumerate}
Thanks to the assumptions (H1) and (H2), the first term of $\mathcal{F}$ can be understood as  
$$
\lim_{\delta\to 0}\int_{\Omega_{\delta}}\dist (x,\partial \Omega) \chi_{\Sigma^u_{\infty}}d\mathcal{H}^1
$$
while the second one is
$$
\lim_{\delta\to 0}\lim_{h\to 0}\sum_{i,j=1}^2\int_{\Omega_{\delta}\setminus I_h(\Sigma^u_{\infty})}\big(\dist(x,\Sigma^u_{\infty})\big)^\alpha d|Du^j_{x_i}|\,.
$$
Observe that these limits exist by monotonicity and therefore they can be computed simply taking the supremum on $h$ and $\delta$.

The first term of $\mathcal{F}$  controls in some sense the fractal behavior of a solution to (\ref{OPGG}) at the boundary of
$\Omega$ and discards maps with an infinite $\mathcal{H}^1$-measure of $\Sigma^u_{\infty}$ far from $\partial \Omega$. The role of the second term is to minimize the spread of the singularities of $Du$ in $\Omega$. It is determinant to deal for example with the case when the fractalisation of singularities takes place only at $\partial \Omega$ (e.g. the square $(-a,a)\times (a,a)$). 
In this case the first term is identically zero  and the second term  performs the selection by choosing the solutions that minimize the weighted length of the jumps of  the gradient. 

In order to minimize $\mathcal{F}$ using the direct methods, we need some compactness properties of minimizing sequences. 
As first step in the next theorem we start focusing on sequences $\{u_n\}$ of maps in $\mathcal S$ with uniformly bounded energy, that is $\mathcal{F}(u_n) \leq C < \infty$ for some constant $C\in \R$.

\begin{thm}
\label{thm:compact} Let $\Omega \subset \R^{2}$ be open and such that there exist $M>0$ and $\delta_{0}>0$ with 
\begin{equation}\label{boundary-reg}
\mathcal{H}^{1}(\partial \Omega_{\delta}) \leq M \;\; ,\;\;\; \forall \, \delta < \delta_{0}. 
\end{equation}
	Let $\{u_n\}_{n\subset \mathbb N}\in \mathcal S$ be such that there exists $0<C<+\infty$ with  
	\[
	\mathcal{F}(u_n) \leq C .
	\]
Then the following holds true:
	\begin{enumerate}
		\item
		There exists  $\Sigma_\infty \subset \overline\Omega$ such that for any $\delta>0$,   
		\begin{enumerate}
		\item
		$\Sigma_\infty \cap \overline\Omega_{\delta}$ has finite $\mathcal{H}^1$-measure;
		\item $(\Sigma_{\infty} \cap \Omega_{\delta}) \cup \partial \Omega_{\delta}$ is a connected set;
		\item
		$(\Sigma^{u_n}_{\infty} \cap \overline\Omega_{\delta})\to  (\Sigma_{\infty} \cap \overline \Omega_{\delta })$ in the Hausdorff metric. 
				\end{enumerate}
		\item
		There exists a solution $u$ to (\ref{OPGG}) such that $u^j_n \weakstar  u^j$ in $W^{1,\infty}_0(\Omega)$; for any $\eps>0$ and $\delta>0$ sufficiently small,	 $u^j_{n_{x_i}} \weakstar  u^j_{x_i}$ in $BV\big(\Omega_{\delta}\setminus I_\varepsilon(\Sigma^\delta_{\infty})\big)	$	
		and in measure in $int[\Omega_{\delta}\setminus I_\varepsilon(\Sigma_{\infty})]$, up to a subsequence. 
		$Du$ is piecewise constant in  $int[\Omega_{\delta}\setminus I_\varepsilon(\Sigma_{\infty})]$.
			\item The following inclusion holds:
		\begin{equation}
		\label{eq:inclusion-01}
		\Sigma^{u}_{\infty}\cap \Omega_{\delta}\subseteq \Sigma_{\infty}\cap \Omega_{\delta}\,.
		\end{equation}
	\end{enumerate}
\end{thm}

\begin{rem}
	We observe that, if $u_n$ is a minimizing sequence for $\mathcal{F}$, the map $u$ obtained in the previous theorem is a good candidate to be a selected solution to (\ref{OPGG}). 
	Indeed, due to \eqref{eq:inclusion-01}, we can say that, roughly speaking, the singular set of the limit map $u$ is smaller than the limit of the singular set of any minimizing sequence, at least far from the boundary of $\Omega$.   
\end{rem}

\begin{proof}
Let $\delta<\delta_{0}$. We divide the proof in three steps, one for each assertion of the theorem.

\vspace{0.1cm}

\boxed{\textrm{\it Proof of (1)}}  Since $u_n\in \mathcal{S}$ is a sequence with uniformly bounded energy,  the first term of $\mathcal{F}$ is uniformly bounded on $u_n$ and this implies that 
\begin{equation}\label{eq:firsttermbounded}
\mathcal{H}^1\big(\Sigma^{u_n}_{\infty} \cap \overline\Omega_{\delta}\big)\leq C_1,
\end{equation}
for some positive constant $C_1$. Therefore, by \eqref{boundary-reg}, for $\delta < \delta_{0}$,  
\[
 \mathcal{H}^1\big((\Sigma^{u_n}_{\infty} \cap \overline \Omega_{\delta}) \cup \partial \Omega_{\delta}\big) \leq C_{1}+M < \infty.
\]
Moreover, by Lemma \ref{lem:closureSigma}, $\Sigma_{\infty}^{u_n}$ is closed. We can then apply the Blaschke's Selection Theorem (see Theorem \ref{Blaschke}) 
to the sequence $\Sigma^{u_n}_{\infty} \cap \overline \Omega_{\delta}$ to ensure the existence of a compact set $\Sigma_{\infty}^{\delta}\subset \overline\Omega_{\delta}$ such that, up to a subsequence,  
 $$
\Sigma^{u_n}_{\infty} \cap \overline \Omega_{\delta} \to \Sigma_{\infty}^{\delta}
 $$ 
in the Hausdorff metric for $n\to \infty$. It follows from the previous convergence that
 \[
 (\Sigma^{u_n}_{\infty} \cap \overline \Omega_{\delta}) \cup \partial \Omega_{\delta} = 
(\Sigma^{u_n}_{\infty} \cap \Omega_{\delta}) \cup \partial \Omega_{\delta} \to \Sigma_{\infty}^{\delta} \cup \partial \Omega_{\delta} =:  \tilde\Sigma_{\infty}^{\delta}
\]
in the Hausdorff metric for $n\to \infty$. Thanks to assumption (H1), by Gol\c ab's  Theorem (see Theorem \ref{thmGolab}), $\tilde\Sigma_{\infty}^{\delta} $ is connected and 
$$
\mathcal{H}^1\big(\tilde\Sigma_{\infty}^{\delta}\cap \Omega_{\delta}\big)\leq \liminf_{n\to \infty}\mathcal{H}^1\big(\Sigma^{u_n}_{\infty} \cap \Omega_{\delta}\big).
$$
From the previous estimate and (\ref{eq:firsttermbounded}), we easily deduce that $\Sigma_\infty^\delta$ has finite $\mathcal{H}^1$-measure. Moreover it is not hard to see that for $0<\delta_{1}\leq \delta_{2}< \delta_{0}$ we have $\Sigma_{\infty}^{\delta_{2}}\subseteq \Sigma_{\infty}^{\delta_{1}}$.
Finally we define \[
\Sigma_\infty:= \bigcup_{\delta>0}\Sigma_\infty^\delta.
\] 
\vspace{0.1cm}
\boxed{\textrm{\it Proof of (2)}}
We first claim that
\begin{equation}\label{L^2mesureto0-a}
\mathcal{L}^2\big(I_\varepsilon\big(\tilde\Sigma^\delta_{\infty}\big)\big)\to \mathcal{L}^2\big(\tilde\Sigma_{\infty}^{\delta}\big) = 0\,.
\end{equation}
To prove it we will apply Theorem \ref{Thm:minkowskylowerbound}. 
$\tilde\Sigma_{\infty}^{\delta}$ is  arcwise connected by Proposition \ref{connexion_by_arcs}, compact by Theorem \ref{Blaschke} and then rectifiable (see Section \ref{subsec:hausdorff}).
For any $x\in \tilde\Sigma_{\infty}^{\delta}$, 
 the $\mathcal{H}^1$-measure of $\tilde\Sigma_{\infty}^{\delta}\cap B(x,\rho)$ is greater than $2\rho$, by the isoperimetric inequality. To conclude it is then sufficient to use $\nu:=\mathcal{H}^1\lfloor{\tilde\Sigma_{\infty}^{\delta}}$ in Theorem \ref{Thm:minkowskylowerbound} and apply the notion of Minkowski content together with the information that $\tilde\Sigma_\infty^\delta$ has finite $\mathcal{H}^1$-measure. From \eqref{L^2mesureto0-a} we get
\begin{equation}\label{L^2mesureto0}
\mathcal{L}^2\big(I_\varepsilon\big(\Sigma^\delta_{\infty}\big)\big)\to \mathcal{L}^2\big(\Sigma_{\infty}^{\delta}\big) = 0\,.
\end{equation}
Since $D u_n$ is allowed to take only a finite number of values, for $j\in\{1,2\}$, $Du^j_n$ is uniformly bounded in $L^{\infty}(\Omega)$. This and the vanishing boundary condition imply that  $\|u^j_n\|_{L^{\infty}(\Omega)}$ is bounded. Therefore, up to a subsequence,
$$
u^j_n \weakstar  u^j \, \,\mbox{in}\, W^{1,\infty}_0(\Omega)
$$  for some $u=(u^1,u^2)\in W^{1,\infty}_0(\Omega,\mathbb{R}^2)$.

Since $u_n\in \mathcal{S}$ is a sequence with uniformly bounded energy, 
the second term of $\mathcal{F}(u_n)$ is uniformly bounded. By the Hausdorff convergence
of 
$\Sigma^{u_n}_{\infty} \cap \overline\Omega_{\delta}$ to $\Sigma_{\infty}^{\delta}$  proved in the previous step, one gets
\[
 \sum_{i,j=1}^2\int_{int\big[\Omega_{\delta}\setminus I_\varepsilon(\Sigma^\delta_{\infty})\big]}\big(\dist(x,\Sigma^{u_n}_{\infty})\big)^\alpha d|(Du^j_{n})_{x_i}|
 \leq C_2,
\]
for some positive constant $C_2$.
Moreover, for a sufficiently large $n$, 
there exists a positive constant $C_3$ such that
  $(\dist(x,\Sigma^{u_n}_{\infty}))^\alpha\geq C_3$ 
  for any $x \in \Omega_{\delta}\setminus I_\varepsilon(\Sigma^\delta_{\infty})$.
Therefore
\[
\sum_{i,j=1}^2 \left\|(u^j_{n})_{x_i}\right \|_{BV\left(\Omega_{\delta}\setminus I_\varepsilon(\Sigma^\delta_{\infty})\right)} \leq C_4
\]
for some positive constant $C_4$. 
By Lemma \ref{lem:SBV-farfromSigma},  $D u^j_n$ is piecewise constant in $\Omega_{\delta}\setminus I_{\varepsilon}(\Sigma^\delta_{\infty})$. Applying Theorem \ref{compattezzapw} in a Lipschitz domain $\tilde \Omega$ with 
\[
\Omega_{\delta}\setminus I_{2\varepsilon}(\Sigma^\delta_{\infty}) \subset\tilde \Omega \subset \Omega_{\delta}\setminus I_\varepsilon(\Sigma^\delta_{\infty}),
\] 
up to a subsequence, we have $\big(u^j_n\big)_{x_i}\to g_i^{j,\varepsilon,\delta}$ in measure on $\Omega_{\delta}\setminus I_{2\varepsilon}(\Sigma^\delta_{\infty})$  as  $n\to \infty$, for some piecewise constant function $g_i^{j,\varepsilon,\delta}$, $i=1,2$. 
The uniqueness of the limit implies that $u_{x_i}=g_i^{j,\varepsilon,\delta}$ on $int\left[\Omega_{\delta}\setminus I_{2\varepsilon}(\Sigma^\delta_{\infty})\right]$. 
Therefore, for $j=1,2$, up to a subsequence,
\[ \left(u^j_{n}\right)_{x_i}\weakstar  u^j_{x_i} \,\, \textrm{in }BV\big(\Omega_{\delta}\setminus I_{2\varepsilon}(\Sigma^\delta_{\infty})\big) \] 
and
\[
\left(u^j_{n}\right)_{x_i}\longrightarrow  u^j_{x_i} \,\,
\textrm{a.e. in } int\big[\Omega_{\delta}\setminus I_{2\varepsilon}(\Sigma^\delta_{\infty})\big]\,.
\]
This being true for every $\varepsilon$ and $\delta$, combined with \eqref{L^2mesureto0} ensures that $Du \in E$ a.e. in $\Omega$. In other words, $u$ is a solution to (\ref{OPGG}). 

\vspace{0.1cm}

\boxed{\textrm{\it Proof of (3)}}
We finally prove (\ref{eq:inclusion-01}) that is equivalent to show that $ [\Sigma_{\infty}\cap \Omega_{\delta}]^{c} \subset [ \Sigma^{u}_{\infty}\cap \Omega_{\delta}]^{c}$. 
Let  $x_0 \notin \Sigma_{\infty}\cap \Omega_{\delta}$, i.e. $x_{0}\in [\Sigma_{\infty}^{\delta}]^{c} \cup \partial \Omega_{\delta}$. 
If $x_{0}\in \partial \Omega_{\delta}$ then $x_0\in [\Sigma_{\infty}^{\delta} \cap  \Omega_{\delta}]^{c}$ since $\partial \Omega_{\delta} \subset [ \Sigma^{u}_{\infty}\cap \Omega_{\delta}]^{c}$.  
If  $x_{0}\in  [\Sigma_{\infty}^{\delta}]^{c}$,  choose $r_0$ such that
\[
r_0<\frac{\dist(x_0,\Sigma_{\infty}^{\delta})}{3}.
\] 
For a sufficiently small $\varepsilon$, $B(x_0,r_0)\cap I_
\varepsilon(\Sigma^\delta_{\infty})$ is empty. 
Since $(\Sigma^{u_n}_{\infty}\cap \overline\Omega_{\delta})$ converges to $\Sigma^\delta_{\infty}$ in the Hausdorff metric,
 $B(x_0,r_0)$ is not contained in $\Sigma^{u_n}_{\infty}\cap \Omega_{\delta}$ for $n$ sufficiently large. 
By the previous point,   $Du^j_n \to Du^j$  in $B(x_0,r_0)$ and $Du$ is piecewise constant in $B(x_0,r_0)$. It follows
that $x_0\notin \Sigma^{u}_{\infty}\cap \Omega_{\delta}$ and the claim is proved.

\end{proof}

If we can assure that the functional $\mathcal{F}$ is not identically $+\infty$ on $\mathcal S$, the previous theorem applied to a minimizing sequence of $\mathcal{F}$ on $\mathcal S$, provides us with a map $u$ that is a candidate to be a minimizer. To ensure that $u$ is indeed a minimizer, it would be needed to prove that $u$ belongs to $\mathcal{S}$ and that the functional $\mathcal{F}$ is lower semicontinuous. This is not an easy task in the whole $\mathcal S$,
so we are led to require an additional condition on the family of solutions to (\ref{OPGG}) on which we minimize $\mathcal{F}$.  
According to Remark \ref{rem:segmentofisarmonica} it seems reasonable to restrict our analysis to solutions such that the accumulations of jumps of $Du$ is created by an accumulation of mass from at least three different sets $\Omega^u_i$. We conjecture that, at least for special sets $E$ and when $\Omega$ has a simple geometry, this restriction could be formulated only in terms of a uniform lower bound on the density of the sets $\Omega^u_i$ (for at least three indexes) at $\mathcal{H}^1$-almost every point of $\Sigma^u_\infty$. Unfortunately for our proof we need a slightly stronger hypothesis that will be precisely stated in Definition \ref{def:uniformbounddensity} and that is motivated by the following theorem.

\begin{thm}
	\label{thm:equivalence}
	Let $\{u_n\}\subset \mathcal{S}$ be a sequence with uniformly bounded energy. 
	Assume that
there exist a constant $c>0$ and for any $\delta>0$ a constant $0<\overline{r}<\delta$ such that for any $u_n$ and $\mathcal{H}^1-a.e. \;x\in \Sigma^{u_n}_{\infty}\cap \Omega_\delta$ we can find at least three indices $i_1,\, i_2,\,i_3\in \{1,\dots, k\}$ with the property that for every  $r<\overline{r}$ we have
\begin{equation}\label{eq:densitaH3}
\mathcal{L}^2\left(B(x,r)\cap \Omega^{u_n}_{i_s}\right )>c r^2\,, \; \; s\in \{1,2,3\}.
\end{equation}
Then the limit function $u$ obtained in Theorem \ref{thm:compact} belongs to $\mathcal S$ and satisfies
	\begin{equation}\label{ugualianzaEinfinito}
	\Sigma^{u}_{\infty} \cap \Omega_{\delta}=  \Sigma_{\infty}\cap \Omega_{\delta}.
	\end{equation}
	Moreover $u$ satisfies 
\begin{equation}\label{eq:densitaH3u}
	\mathcal{L}^2\left(B(x,r)\cap \Omega^{u}_{i_s}\right)>c r^2\,, \; \; s\in \{1,2,3\}
\end{equation}
	for every  $r<\overline{r}$ and \begin{equation}\label{semicontinuitaGiovio}
\mathcal{F}(u)\leq \liminf_{n\to \infty}\mathcal{F}(u_n).
\end{equation}
\end{thm}

\begin{proof}
In Theorem \ref{thm:compact} we already proved  \eqref{eq:inclusion-01}, that is, the inclusion
$\Sigma^{u}_{\infty}\cap \Omega_{\delta} \subseteq  \Sigma_{\infty}\cap \Omega_{\delta} $.
Therefore
we only need to show
the inverse inclusion. 
Arguing by contradiction we suppose the existence of 
$$
y_0 \in (\Sigma_{\infty} \cap \Omega_{\delta})  \setminus  \Sigma^{u}_{\infty}\,.
$$ 
Since $\Sigma^u_\infty$ is closed and $y_{0}\in \Omega_{\delta}$ we can choose $0<R< \min\{\dist(y_{0},\partial \Omega_{\delta}),\dist(y_0,\Sigma^u_\infty)\}$ so that $B(y_0,R)\cap \Sigma_{\infty}^u=\emptyset$ and $B(y_0,R)\subset \Omega_{\delta}$. We recall that, by Theorem \ref{thm:compact}, the set $ (\Sigma_{\infty} \cap \Omega_{\delta}) \cup \partial \Omega_{\delta}$ is connected and then arcwise connected by  Proposition \ref{connexion_by_arcs}. 
Inclusion \eqref{eq:inclusion-01} implies the existence of $y_1\in (\Sigma_{\infty} \cap \Omega_{\delta}) \setminus B(y_0,R)$, and a path lying in $(\Sigma_{\infty} \cap \Omega_{\delta}) \cup \partial \Omega_{\delta}$, joining $y_0$ and $y_1$. By the isoperimetric inequality, we easily deduce that
\begin{equation}
	\label{eq:lenghtboundR}
	\Haus^1\big(\Sigma^\delta_\infty\cap B(y_0,R) \big) \geq R.
\end{equation}
As observed in Theorem \ref{thm:compact}, $Du$ is piecewise constant in $int[\Omega_{\delta}\setminus I_\varepsilon(\Sigma^\delta_{\infty})]$ for any $\varepsilon$. Therefore, by \eqref{eq:lenghtboundR} and Theorem \ref{thm:caccio-structure} we can choose $\overline{y}\in \Sigma^\delta_\infty \cap B(y_0,R)$ such that, up to a permutation of the indexes names,  
\begin{equation}
\label{eq:dens}
	\theta(\Omega^u_i,\overline y)=0 \;\; \forall \, i\in \{1,\dots k-2\}.
\end{equation}
Let $\eps>0$ be sufficiently small, to be chosen later.  Using \eqref{eq:dens} we deduce that there exists $r'>0$ such that
\begin{equation}
	\label{eq:densitylow}
	\frac{\mathcal{L}^2\big( B(\overline y,r ) \cap \Omega^u_i \big)}{\pi r^2} < \eps \;,\;\;\; \forall\, r\leq r' \;,\;\;\forall i\in \{1,\dots k-2\}.
\end{equation}
The Hausdorff convergence of $\Sigma_{\infty}^{u_n}\cap \overline\Omega_{\delta}$ to $\Sigma_{\infty}^{\delta}$, proved in Theorem \ref{thm:compact}, and  assumption  \eqref{eq:densitaH3} ensure us that there exist an index $j\in\{1,\dots,k-2\}$ and a sequence $x_n\in \Sigma_{\infty}^{u_n}$ with $|x_n-\bar y|<\frac{r'}{4}$ and such that
\begin{equation}\label{eq:conseguenzaH3}
\mathcal{L}^2\big( B(x_n,r ) \cap \Omega^{u_n}_j \big) > c r^2 \;,\;\; \forall \, r< \bar r.	
\end{equation}
We now observe that, since $u_n \weakstar  u$ in $W^{1,\infty}_{loc}(\Omega, \R^2)$, as proved in  Theorem \ref{thm:compact}, we have
\begin{equation}
	\label{eq:weakmeas_01}
	\lim_{n\to \infty} \mathcal{L}^2\big( \Omega^{u}_j \cap \Omega^{u_n}_i \cap B(y_0,R) \big) =0 \;\; \textrm{ for any }\, i\not= j
\end{equation} 
and 
\begin{equation}
	\label{eq:weakmeas_02}
	\lim_{n\to \infty} \mathcal{L}^2\big( \Omega^{u}_i \cap \Omega^{u_n}_i \cap B(y_0,R) \big) = \mathcal{L}^2\big( \Omega^{u}_i \cap B(y_0,R) \big) \;,\;\;\forall i\in \{1,\dots k\}
\end{equation}
Now, using \eqref{eq:weakmeas_01}, choose $r_0= \min\left\{ \bar r, \frac{r'}{4}\right \}$ and $\bar n$ sufficiently large such that for $n>\bar n$, 
\[
\mathcal{L}^2\big( ( \Omega^{u_n}_j\setminus \Omega^u_j) \cap B(y_0,R) \big) < \frac{c r_0^2}{10}\,.
\]
By (\ref{eq:conseguenzaH3}) and (\ref{eq:weakmeas_02}) and the inclusion $B(x_n,r_0)\subset B(\bar y,r')$, we have that
$$
\mathcal{L}^2\big(  B(\bar y,r') \cap \Omega^u_j  \big) >\mathcal{L}^2\big(  B(x_n,r_0) \cap \Omega^u_j  \big) > \frac{c r_0^2}{2}\,.
$$
The last inequality gives a contradiction with \eqref{eq:densitylow} if we choose $\eps$ sufficiently small, thus proving the claim. 
We explicitly note that (\ref{ugualianzaEinfinito}) guarantees that $\mathcal{H}^1(\Sigma_{\infty}^{u})$ is locally finite and $(\Sigma_u\cap \Omega_{\delta})\cup \partial \Omega_\delta$ is connected.

We are now going to check (\ref{eq:densitaH3u}) for $u$. To this aim, fix $\delta$ and let $x_0\in \Sigma^u_\infty\cap \Omega_\delta$ and $x_n\in \Sigma^{u_n}_\infty\cap \Omega_\delta$ be such that $x_n \to x_0$. Up to extracting a subsequence, we can find $i_1,\, i_2,\,i_3\in \{1,\dots, k\}$ with the property that for every  $r<\overline{r}$ we have, independently on $n$,
\[
\mathcal{L}^2(B(x_n,r)\cap \Omega^{u^n}_{i_s})>c r^2 \;, \; s\in \{1,2,3\}.
\]
Now, for sufficiently large $n$ we have 
\[
\begin{split}
\mathcal{L}^2\big( B(x_0,r)\cap \Omega^u_{i_s}\big)   \geq & \mathcal{L}^2\big( B(x_n,r)\cap \Omega^{u^n}_{i_s}\big)  - \mathcal{L}^2\big( \Omega^{u_n}_{i_s} \triangle \Omega^u_{i_s} \big) \\
 &   - \mathcal{L}^2\big( B(x_n,r) \triangle B(x_0,r)\big) \\
   > &  c r^2 - \frac{2}{n}.
\end{split}
\]
The limit as $n\to \infty$ proves (\ref{eq:densitaH3u})  for $u$. 

The theorem will be proved once we show \eqref{semicontinuitaGiovio}.
To prove the convergence of the first term of $\mathcal F(u_n)$, i.e.
\[
 \int_{\Omega}\dist(x,\partial \Omega) \chi_{\Sigma^{u_n}_{\infty}}d\mathcal{H}^1
\]
one can apply Theorem \ref{thm4.1Giacomini}.
So we are left with the study of 
\[
\sum_{i,j=1}^2 \int_{\Omega}
\big(\dist(x,\Sigma^{u_n}_{\infty}\big)^{\alpha} d|D(u_n)^j_{x_i}|.
\]
To this aim, fix the indexes $i$ and $j$, $\delta$, $\rho>0$  and let
\[
A^u_{\delta, \rho}:=\int_{\Omega_\delta \setminus I_{\rho}(\Sigma^u_{\infty})}
\big(\dist(x,\Sigma^u_{\infty})\big)^{\alpha} d|Du^j_{x_i}|.
\]
We observe that $A^u_{\delta, \rho}$ is decreasing with respect to $\delta$ and $\rho$.
Therefore, being 
\[
\mathcal{F}(u)=\lim_{\delta\to 0}\lim_{\rho\to 0} A^u_{\delta, \rho}=\sup_{\delta>0}\sup_{\rho>0}A^u_{\delta,\rho}=
\lim_{\rho\to 0}\lim_{\delta\to 0} A^u_{\delta, \rho},
\]
the inequality (\ref{semicontinuitaGiovio}) reduces to
\[
\sup_{\delta>0,\rho>0}A^u_{\delta, \rho}
\leq \liminf_{n\to \infty}\sup_{\delta>0,\rho>0}A^{u_n}_{\delta, \rho}
=
\sup_n\inf_{k>n}\sup_{\delta>0,\rho>0}A^{u_k}_{\delta, \rho},
\]
with the obvious meaning for the notation $A^{u_n}_{\delta, \rho}$. Up to a subsequence, it is therefore sufficient to prove that
\[
A^u_{\delta, \rho}\leq
\lim_{k\to \infty} A^{u_{n_k}}_{\delta_k, \rho_k},
\]
with $\delta_k\to 0$ and $\rho_k\to 0$. 
By Proposition \ref{prop:crit-compact}, since $Du^j_{n_{x_i}}\weakstar Du^j_{x_i}$, we have that
\begin{equation}\label{after_Reshetnyiak}
\int_{\Omega_\delta \setminus I_{\rho}(\Sigma^u_{\infty})}
\big(\dist(x,\Sigma^u_{\infty})\big)^{\alpha} d|Du^j_{x_i}|
\leq \liminf_{n\to \infty}
\int_{\Omega_\delta \setminus I_{\rho}(\Sigma^u_{\infty})}
\big(\dist(x,\Sigma^u_{\infty})\big)^{\alpha} d|(Du^j_{n})_{x_i}|\,.
\end{equation}
Up to a subsequence, we can assume that the liminf in the right hand side of the above inequality is a limit.

It is easy to prove that $\dist(x,\Sigma^{u_n}_{\infty})\to \dist(x,\Sigma^u_{\infty})$ in $L^{\infty}(\Omega_{\delta})$, since
$\Sigma_{\infty}^{u_n}\cap \overline\Omega_{\delta} \to \Sigma_{\infty}^{u}\cap \overline\Omega_{\delta}$ in the Hausdorff metric, by Theorem \ref{thm:compact}.

Now, one can estimate terms appearing in the right hand side of (\ref{after_Reshetnyiak}) as
\begin{equation}\label{after_after_Reshetnyiak}
\begin{split}
\int_{\Omega_\delta \setminus I_{\rho}(\Sigma^u_{\infty})} & 
\big(\dist(x,\Sigma^u_{\infty})\big)^{\alpha}  d\,|(Du^j_{n})_{x_i}|
 \\
&\leq \int_{\Omega_\delta \setminus I_{\rho}(\Sigma^u_{\infty})}
\left|\big(\dist(x,\Sigma^u_{\infty})\big)^{\alpha}-\big(\dist(x,\Sigma^{u_n}_{\infty})\big)^{\alpha}\right| d\,|(Du^j_{n})_{x_i}|
\\
& + \int_{\Omega_\delta \setminus I_{\rho}(\Sigma^u_{\infty})}
\big(\dist(x,\Sigma^{u_n}_{\infty})\big)^{\alpha} d\,|(Du^j_{n})_{x_i}|.
\end{split}
\end{equation}

It is clear that the first term of the right hand side of the above inequality
tends to 0, as $n\to \infty$, since $Du_n$ is uniformly bounded in 
$\Omega_\delta \setminus I_{\rho}(\Sigma^u_{\infty})$
and the integrand tends to 0 uniformly, as already pointed out.
Let us now comment on the second term of (\ref{after_after_Reshetnyiak}), that is,
$$
\int_{\Omega_\delta \setminus I_{\rho}(\Sigma^u_{\infty})}
\big(\dist(x,\Sigma^{u_n}_{\infty})\big)^{\alpha} d|(Du^j_{n})_{x_i}|\,.
$$
The Hausdorff convergence of $\Sigma^{u_n}_{\infty}\cap \overline\Omega_{\delta}$ to $\Sigma^{u}_{\infty}\cap \overline\Omega_{\delta}$ implies the existence of $\overline{n}\in \mathbb{N}$, such that $d_{\mathcal{H}}(\Sigma^{u_n}_{\infty}\cap \overline\Omega_{\delta}, \Sigma^{u}_{\infty}\cap \overline\Omega_{\delta})<\frac{\rho}{10}$ 
for any $n>\overline{n}$. Therefore
$$
\Omega_\delta \setminus I_{\rho}(\Sigma^u_{\infty})
\subset
\Omega_\delta \setminus I_{\frac{\rho}{3}}(\Sigma^{u_n}_{\infty})\,.
$$
This implies that
$$
\int_{\Omega_\delta \setminus I_{\rho}(\Sigma^u_{\infty})}
\big(\dist(x,\Sigma^{u_n}_{\infty})\big)^{\alpha} d|(Du^j_{n})_{x_i}|
\leq
\int_{\Omega_\delta \setminus I_{\frac{\rho}{3}}(\Sigma^{u_n}_{\infty})}
\big(\dist(x,\Sigma^{u_n}_{\infty})\big)^{\alpha} d|(Du^j_{n})_{x_i}|\,.
$$
This allows us to conclude the proof of (\ref{semicontinuitaGiovio}).
\end{proof}

The previous result motivated us to consider the following definition

\begin{defn}\label{def:uniformbounddensity}
Given $c>0$ a fixed constant and $\phi : (0,\infty) \to (0,\infty)$ a positive function, a solution $u\in \mathcal{S}$ to \eqref{OPGG} is said to be  {\it $(c,\phi)$-uniformly lower bounded in density } for any $\delta>0$, $\mathcal{H}^1-a.e. \;x\in \Sigma^{u}_{\infty}\cap \Omega_\delta$ we can find at least three indices $i_1,\, i_2,\,i_3\in \{1,\dots, k\}$ with the property that for every  $r<\phi(\delta)$ we have
\begin{equation}\label{eq:densitaH3-def}
\mathcal{L}^2(B(x,r)\cap \Omega^{u}_{i_s})>c r^2\,, \; \; s\in \{1,2,3\}.
\end{equation}
We define $\mathcal{S}_{c}^{\phi}\subset \mathcal{S}$ to be the family of all the maps $u\in \mathcal{S}$ that are $(c,\phi)$-uniformly lower bounded in density. 
\end{defn} 

Even if the requirement of $(c,\phi)$-uniformly lower bound in density for a class of solutions seems quite restrictive, Lemma \ref{lemma:H3pv} allows us to prove that for the interesting case of system \eqref{OP} there exist classes of solutions satisfying this property independently on the geometry of the domain $\Omega$. This is the content of the next lemma.

\begin{lem}
\label{lem:S_E-notempty}
Let $E:=\{\pm A_i\;:\; i\in\{1,\dots,4\}\}$ where the $A_i$ are the matrices  defined in \eqref{eq:8matrici}. Consider the constant $c$ given by Lemma \ref{lemma:H3pv}. Then the family of maps $u\in \mathcal S$ that satisfy the $(c,\phi)$-uniformly lower bound in density with $\phi(\delta)= \delta/4$ is not empty.

\end{lem}
\begin{proof}
Let $\Omega$ be any domain and use a Vitali covering of $\Omega$ made up of squares, with the sides parallel to the axes, defined inductively as follows. 
Let $\varepsilon_n:=2-\sum_{i=0}^n \frac{1}{2^i}$ and define, for $n\geq 0$,
$$
\Omega_n:=\{x\in \Omega \,:\, \dist(x, \partial \Omega)>\varepsilon_n\}\,.
$$  
\begin{enumerate}
\item
For $n=0$, we consider a dyadic decomposition of $\mathbb{R}^2$, $\mathcal{D}_0$, with the diagonals of the squares smaller than $\varepsilon_0=1$. 
Let $\{Q^0_i\}_{i\in I_0}\subset \mathcal{D}_0$ be the family of the squares in $\mathcal{D}_0$ with non empty intersection with $\Omega_0$. 
Clearly we have 
$$
\Omega_0 \subset \bigcup_{i\in I_0}Q^0_i \subset \Omega\,.
$$

\item
For any $n\geq 1$ we consider a dyadic decomposition of $\mathbb{R}^2$, $\mathcal{D}_n$, with the greatest diagonal of the squares smaller than $\varepsilon_n$. 
Let $\{Q^n_i\}_{i\in I_n}\subset \mathcal{D}_n$  be the family of the squares in $\mathcal{D}_n$ with non empty intersection with $\Omega_n\setminus \bigcup_{i\in I_{n-1}}Q^{n-1}_i$. 
As before we have 
$$
\Omega_n \subset \bigcup_{j=0}^n\bigcup_{i\in I_j}Q^j_i \subset \Omega\,.
$$
\end{enumerate}
In each square of the  family  $\mathcal{D}:=\{Q^j_i\}_{j\in \mathbb{N}, i\in I_j}$,  we consider 
the solution of subsection \ref{subsec:DMP-construction}. Therefore the set $\Sigma^u_{\infty}$ will be a subset of the union of all the boundaries of the squares belonging to the Vitali covering of $\Omega$.  

It is clear that $\Sigma^u_{\infty}$ satisfies the  condition that  
$(\Sigma^u_{\infty}\cap \Omega_{\delta})\cup \partial \Omega_\delta$ is connected for any $\delta$ such that $\Omega_\delta \not= \emptyset$.
We are going to prove that $\Sigma^u_{\infty}$ is locally of finite $\mathcal{H}^1$-measure in $\Omega$.
Indeed, 
 for any $n\in \N$ there exist only finitely many squares in $\mathcal{D}$ which intersect $\Omega_n$.
Moreover, 
as $\Omega_n \nearrow \Omega$, for any compact $K\subset \Omega$, there exists $n$ such that $K\subset \Omega_n$. 
This implies that 
the length of the union of the boundaries of the squares in $\mathcal{D}$ that intersect $K$ is finite.
To prove (\ref{eq:densitaH3}), we recall that we deal with points in $\Sigma^u_{\infty}\cap \Omega_\delta$. 
Since the length of the diagonals of the squares of the previous Vitali covering in $\Omega_\delta$ is uniformly bounded from below,  it is sufficient to use Lemma \ref{lemma:H3pv}.
\end{proof}

We end this section by observing that combining the results of Theorems \ref{thm:compact} and \ref{thm:equivalence} it is easy to derive the following result.
\begin{cor}
	\label{cor:well-posedness}
	Assume that there exists at least one $u\in \mathcal S_{c}^{\phi} \subset \mathcal S$ such that $\mathcal F(u) < \infty$. Then the variational problem
\[
\inf \big\{ \mathcal{F}(u) \, ,\; u\in \mathcal S_{c}^{\phi}  \big\}
\]  
is well-posed and has a solution $u\in \mathcal S_{c}^{\phi}$.
\end{cor}

\section{Application to the "orthogonal" system}\label{sec:funzionalefinito}

In this section we deal with the set $E$ 
composed by the eight matrices  defined in (\ref{eq:8matrici}). Our aim is to use the minimization of $\mathcal{F}$ as a selection criterion for solutions of \eqref{OPGG}. To apply the results of the previous sections we need to show that the variational problem is well posed. This can be done if we impose some conditions on the domain $\Omega$.

 Some of the arguments in this section will be in the spirit of \cite{Croce:kw} to which we refer for more details when it will be needed. As we will see, an important role will be played by the geometry of $\Omega$. We start giving some definitions. 
\begin{defn}
Let  $h: [a,b]\to \R$  be a $C^1([a,b])$ function with $h'(t)<0$ for every $t\in [a,b]$, we call 
\[
T_{h}:=\left\{(s,t)\in \R^2: a\leq s \leq b, h(b)\leq t\leq h(s)\right\}
\]
 a \emph{triangular domain}. 
\end{defn}
The class of all  triangular domains will be denoted by $\mathcal{T}$. We write $T$ instead of $T_h$ when the definition of the function $h$ is clear by the context.

\begin{defn}
\label{defn:compatible-triang}
Let $T_h$ be a triangular domain. If $\alpha>0$ satisfies 
\[
2\left[\max\left\{\frac{1}{1+\frac{1}{c_1}},\frac{1}{1+c_2}\right\}\right]^{\alpha+1}<1,
\] 
where
\[
-c_1=\min\limits_{x_1\in [a,b]} h'\,,\,\,\,\,-c_2=-\max\limits_{x_1\in [a,b]} h',
\]
then we will refer to $T_h$ as a \emph{$\alpha$-compatible triangular domain}. 
\end{defn}

 \begin{defn}\label{defn:compatibledomain}    
 A bounded connected Lipschitz set $\Omega\subset \R^{2}$ is a  \emph{$\alpha$-compatible domain} if it can be covered by a finite number of rectangles 
 and a finite number of rotations of angle multiple of $\pi/2$ of $\alpha$-compatible triangular domains, say $R[T_{h_k}]$ for $k\in I \subset \N$ and $R$ a rotation,  with mutually disjoint interior and with the graphs of the functions $h_k$ lying on $\partial \Omega$ (cfr. for example Figure \ref{fig:compatible_domain}).
\end{defn}

\begin{figure}[h]
	\centering
	\includegraphics[scale=0.5]{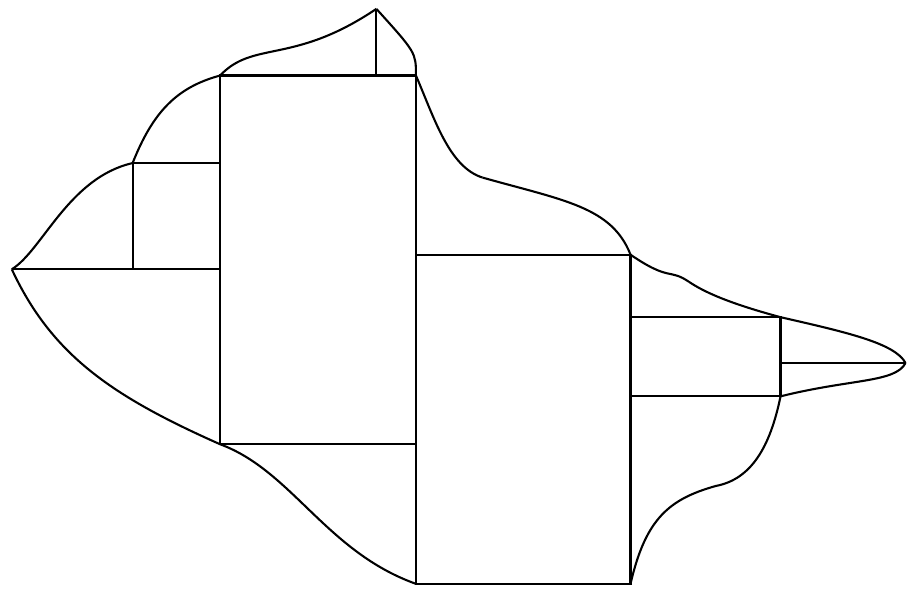}
	\caption{A compatible domain}
	\label{fig:compatible_domain}
	\end{figure}

It is not difficult to verify that a polygon is indeed a $\alpha$-compatible domain. In the sequel we will use the notation $\mathcal{S}_c$ to identify the subset of $\mathcal S$ whose elements are solutions of \eqref{OPGG} that satisfy the $(c,\phi)$-uniformly lower bound in density (according to Definition \ref{def:uniformbounddensity}) with $\phi(\delta)= \delta/4$, where $c$ is given by Lemma \ref{lemma:H3pv}. We explicitly observe that $\mathcal S_c$ is not empty thanks to the Lemma \ref{lem:S_E-notempty}. Now we can state the main theorem of this section.

\begin{thm}
\label{thm:main}
Let $\Omega$ be a $\alpha$-compatible domain. Then the variational problem
\[
\inf \big\{ \mathcal{F}(u) \, ,\; u\in \mathcal{S}_c \big\}
\]  
is well-posed and has a solution, i.e. there exists a minimizer, $u_o\in \mathcal{S}_c$, of $\mathcal{F}(u)$.
\end{thm}

The rest of the section will be devoted to the proof of the previous theorem. Thanks to Corollary \ref{cor:well-posedness}, we just need to exhibit a solution to (\ref{OPGG}) belonging to $\mathcal{S}_c$ for which $\mathcal{F}$ is finite.
The solutions will be defined thanks to explicit coverings of $\Omega$ made up of squares on which we consider  the  solution of Section \ref{subsec:DMP-construction},  as a  {\it building block}. 

As it will be clear in the sequel, it is sufficient to explicitly construct the solutions only on rectangles with sides parallel to the axes and on $\alpha$-compatible triangular domains. Indeed the desired solution for a general $\alpha$-compatible domain can be defined by patching together the solutions associated to the rectangles and to the $\alpha$-compatible triangles whose union gives the domain $\Omega$ according to Definition \ref{defn:compatibledomain}. In the following subsections we will give the detailed constructions in the rectangles, in the $\alpha$-compatible triangles and the general $\alpha$-compatible domains separately.

 We recall that the functional $\mathcal{F}$ is defined by
\begin{equation}
	\label{eq:functionalalpha}
	\begin{split}
\mathcal{F}(u) & =\int_{\Omega}\dist(x,\partial \Omega) \chi_{\Sigma^u_{\infty}}d\mathcal{H}^1+\sum_{i,j=1}^2\int_{\Omega}\big(\dist(x,\Sigma^u_{\infty})\big)^\alpha \,d|Du^j_{x_i}|\\ 
	 & = \mathcal{F}^1(u)+  \sum_{i,j=1}^2 \mathcal F^2_{ij}(u)\,.
	\end{split}
\end{equation}

\subsection{Estimate for rectangular domains}
	\label{subsec:rectangular}

We start our analysis with the case of a square of side $a$:
$$
\Omega=\mathcal{Q}:=\left(-\frac a2,\frac a2\right)^2\,.
$$  
We consider the map $u_{\mathcal Q}(x)=\frac a2 p_v\left(\frac{2}{a}x\right)$ where $p_v$ has been defined in Section \ref{subsec:DMP-construction}. We will denote, with a slight abuse of notation, the squares corresponding, after the rescaling, to $Q_{k,j}$ defined in the construction of $p_v$  with the same symbols. By construction  $\Sigma^u_\infty = \partial \mathcal Q$ and therefore the first term of the functional $\mathcal F$ vanishes. To treat the second term of $\mathcal{F}$, we only have to bound for any $i,j\in\{1,2\}$ the functional $\mathcal F^2_{ij}(u)$

To this aim we start by recalling some simple consequences of the structure of $u$. Let 
\[
T=\{(x_1,x_2):0\leq x_1\leq x_2\leq a/2\}
\] 
and, for $x_n=\frac a2-\frac{a}{2^{n+1}}$, consider for any $n\in \N$ the set
\[
S_n= T \cap \big\{ (x_1,x_2) : x_n < x_2 < x_{n+1} \big\}\,.
\] 
Note that the width of $S_n$ is equal to $l_n:=\frac{a}{2^{n+2}}$ and that  
$S_n$ is covered by at most $2^{n+1}$ squares  $Q_{n,j}$ of side equal to $l_n$. Moreover, since the distance of any point $y\in S_n$ from $\partial \mathcal{Q}$ is the distance of $y$ from the vertical line $x_1=\frac a2$, we clearly have
\[
\dist(y,\partial \mathcal{Q}) < \frac{a}{2^{n+1}} = 2 l_n.
\]
Finally, we recall that for any square $Q_{n,j}$ of the construction with side $l_n$, the part of the support of $|Du^j_{x_i}|$ that intersects $Q_{n,j}$ is contained in a uniformly bounded finite number of segments of length less than $4 l_n$ (namely the boundary, the diagonals and the intersection with the lines parallel to the axes passing through the center of the square). It follows that there exists a universal constant $c$ such that
\[
\Haus^1 \big(\operatorname{supp}|Du^j_{x_i}|_{\lfloor{\overline{Q_{n,j}}}} \big) \leq c \,l_n. 
\]
 
Thanks to the previous considerations, in any given square $Q_{n,j}\subset S_n$ of side $l_n$ one has
 \begin{equation}
	\label{eq:1squareEst}
	\int_{\overline{Q_{n,j}}}\big(\dist(x,\partial \mathcal{Q})\big)^\alpha d\,|Du^j_{x_i}| \leq
	(2 l_n)^{\alpha}\int_{\overline{Q_{n,j}}} d\,|Du^j_{x_i}|
	\leq 
	(2 l_n)^{\alpha}\,
c\, l_n 
	\leq 2^\alpha c \, l_n^{\alpha+1}.
\end{equation}

By the symmetry properties  of the  vectorial pyramid $p_v$, in order to obtain an estimate of $\mathcal F^2_{ij}(u)$, we only need to bound the integral on the set $T$. Using the estimate \eqref{eq:1squareEst} and the definition of $l_n$, for $\alpha>0$, we have
\begin{equation}\label{eq:stimaquadratodilatatoDMP}
\begin{split}\int_{T}\big(\dist(x,\partial \mathcal{Q})\big)^\alpha \,d|Du^j_{x_i}| & = \sum_n
\int_{S_n}\big(\dist(x,\partial \mathcal{Q})\big)^\alpha \,d|Du^j_{x_i}| \\ 
  & \leq  K a^{\alpha+1} \sum_{n} 
\frac{1}{2^{n\alpha}}
\leq C a^{\alpha+1}\,,
\end{split}
\end{equation}
thus proving the required estimate for a squared domain.

In the case of a general rectangle
$$
\Omega = \mathcal R_{ab} := \left(-\frac a2,\frac a2\right)\times \left(-\frac b2,\frac b2\right)
$$ 
with $a > b >0$,
we will consider an explicit covering made up of squares for which we can bound the functional $\mathcal{F}$. We cover $\mathcal{R}_{ab}$ with a sequence of squares $\{Q_i\}$ choosing $Q_i$ as the largest square contained in $\mathcal{R}_{ab}\setminus \bigcup_{h=1}^{i-1} Q_h$ and with minimal value of the components (cfr. Figure \ref{fig:rectanglefinale}). We explicitly observe that, depending on the commensurability of $a$ and $b$, we could have only a finite collection of $Q_i$ or an infinite ones. In any case, it is clear that, denoted by $l_i$ the length of the side of $Q_i$, we have 
\begin{equation}
\label{eq:rectanlge-lenght}
\sum_{i} l_i= a+b.
\end{equation} 
We define the candidate function $u$ by defining it in any square $Q_i$ with the same construction we did before in the general square, suitable translated. 
Observe that  $\dist(x,\partial \mathcal{R}_{ab})$ and 
$\dist(x,\partial Q_i), i\in \N$, are bounded functions.

The desired bounds for $\mathcal F(u)$ follow easily from \eqref{eq:stimaquadratodilatatoDMP} and \eqref{eq:rectanlge-lenght}. Indeed,  for any given square $Q_i$ of the construction with side $l_i$ we have $\Sigma^u_\infty \cap \overline{Q_i} \subseteq \partial Q_i$ and consequently $\Haus^1(\Sigma^u_\infty \cap \overline{Q_i})\leq 4 l_i$. Therefore  
\begin{equation}
\label{eq:haus_boind_rect}	
\Haus^1(\Sigma^u_\infty) \leq \sum_i \Haus^1(\Sigma^u_\infty \cap \overline{Q_i}) \leq 4 \sum_i l_i = 4(a+b).
\end{equation}
Estimate \eqref{eq:haus_boind_rect} clearly implies an upper bound for $\mathcal F^1(u)$. Concerning the second term, we note that \eqref{eq:stimaquadratodilatatoDMP} implies that
\[
\mathcal{F}^2_{ij}(u) \leq C \sum_i (l_i)^{\alpha+1}
\]
and the last term is finite for any $\alpha\geq 0$ thanks to \eqref{eq:rectanlge-lenght}.

\newcommand{\rettangolo}{
	\begin{tikzpicture}[x=5mm,y=5mm]
	\draw(-10.5,3.2)--(10.5,3.2);
	\draw(-10.5,-3.2)--(10.5,-3.2);
	\draw(10.5,-3.2)--(10.5,3.2);
	\draw(-10.5,-3.2)--(-10.5,3.2);
	\draw(-4.1,-3.2)--(-4.1,3.2);
	\draw(2.3,-3.2)--(2.3,3.2);
	\draw(8.7,-3.2)--(8.7,3.2);
\draw(8.7,-1.4)--(10.5,-1.4);
\draw(8.7,0.4)--(10.5,0.4);
\draw(8.7,2.2)--(10.5,2.2);
\draw(9.7,2.2)--(9.7,3.2);
\draw(9.7,2.9)--(10.5,2.9);
\draw(10,2.9)--(10,3.2);
\draw(10.3,2.9)--(10.3,3.2);
\draw (-6.5,0)node[left]{{$Q_1$}};
\draw (0,0)node[left]{{$Q_2$}};
\draw (6.5,0)node[left]{{$Q_3$}};
\draw (10.5,-2.2)node[left]{{$Q_4$}};
\draw (10.5,-0.5)node[left]{{$Q_5$}};
\draw (10.5,1.2)node[left]{{$Q_6$}};
\draw (10,2.7)node[left]{{$Q_7$}};
		\end{tikzpicture}
}
\begin{figure}[h]
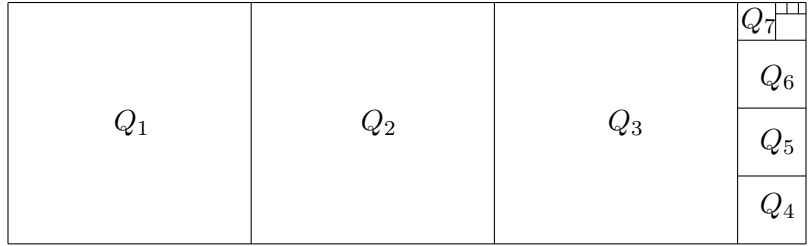

	\centering
	\rettangolo
	\caption{A covering of a rectangle}
	\label{fig:rectanglefinale}
	\end{figure}

\subsection{Estimate for triangular domains}
\label{subsec:triangular}

Given a $\alpha$-compatible triangular domain $T_h\in \mathcal{T}$,  we cover it with a countable number of squares using the construction defined in \cite{Croce:kw}, that we recall here for the reader's convenience.

We start by introducing three operators defined on $\mathcal{T}$. For a given 
$$T=\{(x_1,x_2): a \leq x_1 \leq b, h(b)
\leq x_2 \leq h(x_1)\}\in \mathcal{T}\,,$$ 
 let $x_1^0$ be such that
$h(x_1^0)=x_1^0+h(b)-a$ and define
\[
\begin{array}{lcl}
 \displaystyle  q(T) & := &  \displaystyle  (a,x_1^0)\times (h(b),h(b)+x_1^0-a); \\
  u(T) & := & \displaystyle  \{(x_1,x_2)\in T: a<x_1<x_1^0, h(b)+x_1^0<x_2<h(x_1)\}; \\
  r(T) & := &  \displaystyle  \{(x_1,x_2)\in T: a+x_1^0<x_1<b, h(b)<x_2<h(x_1)\}.
\end{array}
\]
We explicitly observe that $u$ and $r$ have values in $\mathcal{T}$
while $q$ maps any triangular domain $T_h$ to a square contained in $T_h$.
(see Figure \ref{fig-01}). 

\begin{figure}
\label{fig-01}
\centering
\includegraphics[width=7.5cm]{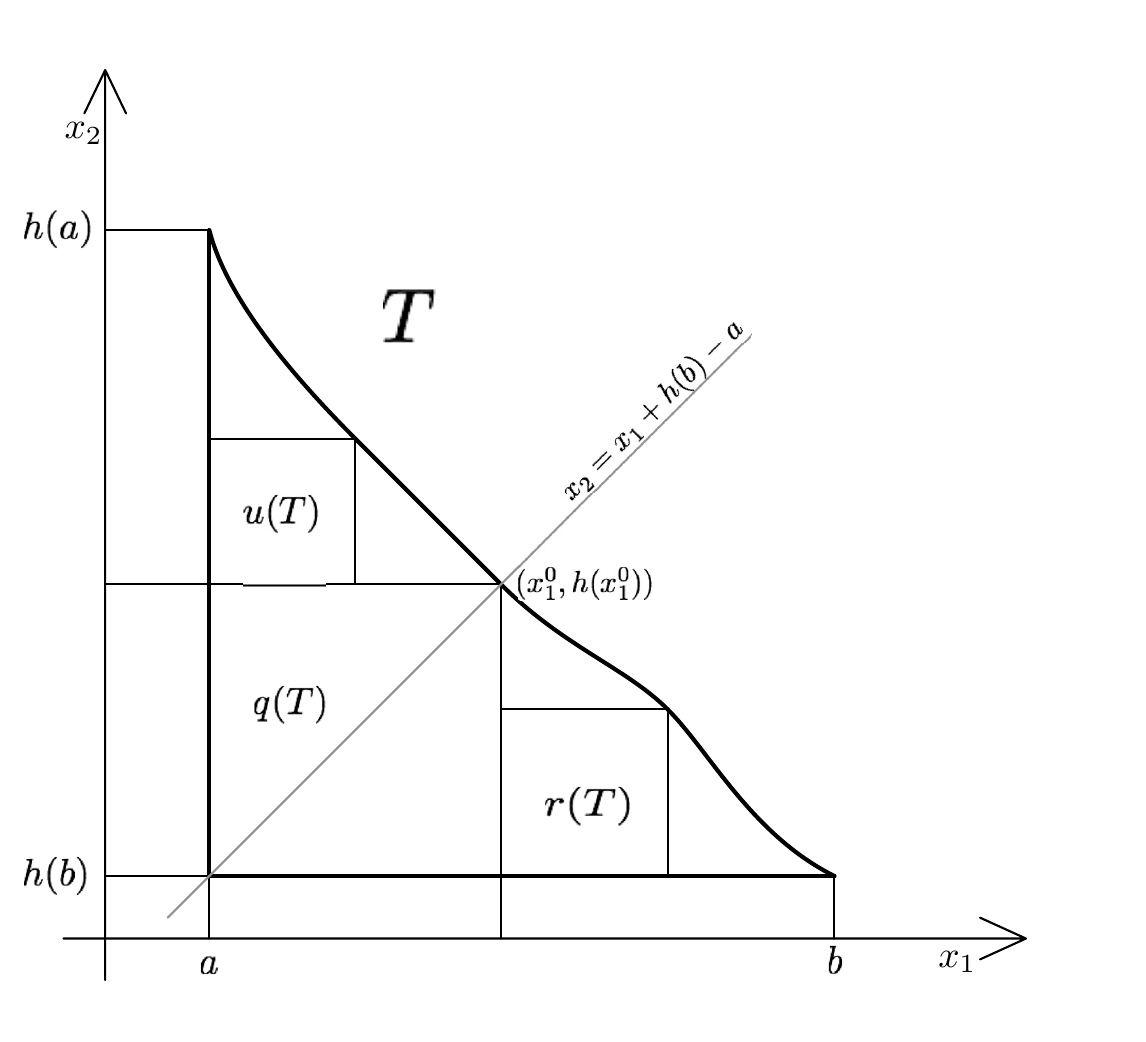}
\caption{definition of the operators $q$, $u$ and $r$.}
\end{figure}

\begin{defn}[Covering of $T$]\label{definition_covering}
Let for $m \in \N$
\[
S_m:=\left\{ \sigma=(\alpha_1,\dots,\alpha_m) \;:\;\alpha_i\in \{u,r\} \right\},
\]
be the set of all the $m$-permutations of the two letters $u$ and
$r$. For $\sigma\in S_m$, using the notation 
\[
\sigma(T)=\alpha_1 \circ \alpha_2 \circ \cdots \circ \alpha_m(T),
\]
we set
\[
Q_T^{m,\sigma}= q(\sigma(T))\;\;;\;\;\;\sigma \in S_m.
\]
We finally define the following family of squares contained in $T$:
\[
\mathcal{Q}(T):=\{Q_T^{m,\sigma} \;:\; m\in \N\cup \{0\} \;,\;\;\sigma\in S_m\}.
\]  
\end{defn}
\begin{rem}\label{rem:vitaliSEE}
It  may be useful to think of $\mathcal{Q}(T)$ as being constructed in steps, starting from $m=1$ and adding at 
step $m$ the squares $Q_T^{m,\sigma}$ with $\sigma\in S_m$. Since the
cardinality of $S_m$ is equal to $2^m$, we add $2^m$ squares at the $m$-th
step. Therefore the first steps of the construction are (see Figure \ref{fig-02}):

\vspace{0.3cm}
\begin{tabular}{lcll}
\vspace{0.2cm}
{\it Step 0.} & We start with &   $Q^0_T=q(T)$ ; & \\ 
\vspace{0.2cm}
{\it Step 1.} & we add & $Q_T^{1,u}=q(u(T))$ \;and &  $Q_T^{1,r}=q(r(T))$ ; \\ \vspace{0.2cm}
{\it Step 2.} &  we add &  $Q_T^{2,(u,u)}=q(u(u(T)))$, &
$Q_T^{2,(u,r)}=q(u(r(T)))$, \\
\vspace{0.2cm}
 & &  $Q_T^{2,(r,u)}=q(r(u(T)))$, & $Q_T^{2,(r,r)}=q(r(r(T)))$ \\
\vspace{0.2cm}
{\it Step 3.} &  we add & $Q_T^{3,(u,u,u)}=q(u(u(u(T))))$, &
$Q_T^{3,(u,u,r)}=q(u(u(r(T))))$, \\
 &  & $\cdots$ & $\cdots$ 
\end{tabular}

\begin{figure}
\label{fig-02}
\centering
\includegraphics[width=7.5cm]{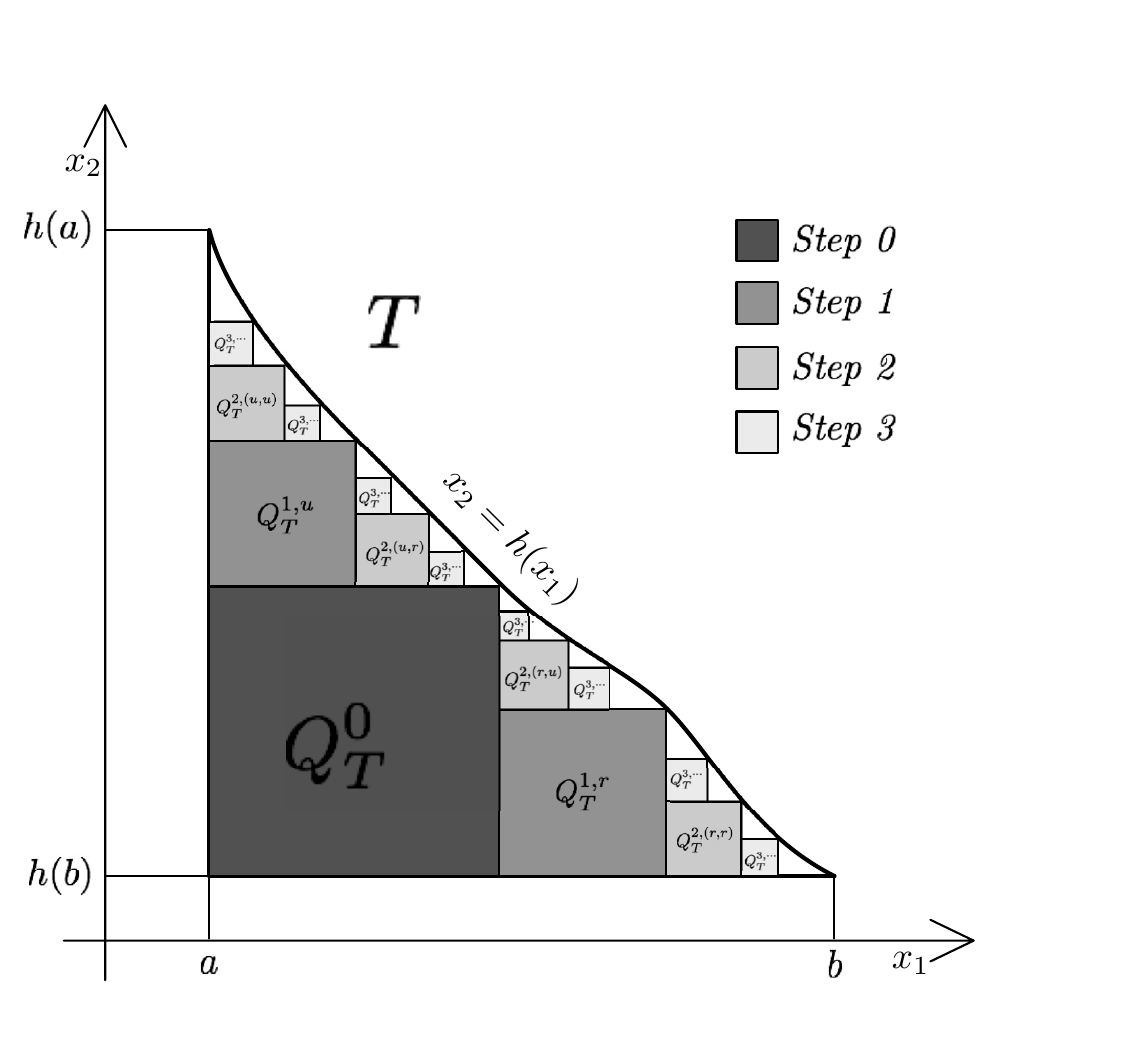}
\caption{construction of $\mathcal{Q}(T)$.}
\end{figure}
\end{rem}

Now we can define our candidate function $u:T \to \mathbb R^2$ prescribing it on any square of the type $Q_T^{m,\sigma}$ as the building block defined in Section \ref{subsec:DMP-construction} suitably rescaled and translated.
We are now going to prove that $\mathcal{F}(u)$ is finite. 
The first term of $\mathcal{F}$, $\mathcal{F}^1(u)$  
can be be bounded performing the same analysis done in \cite[Proposition 4.4]{Croce:kw}. The only difference is that instead of measuring the discontinuity of the gradient, we are measuring  $\mathcal{H}^1(\Sigma^u_\infty)$, that is indeed contained in the union of the boundaries of the squares and makes our analysis applicable. Observe that as in \cite{Croce:kw} we can consider the distance from the graph of the function $h$ defining the triangular domain, instead of the entire boundary of $T_h$. This does not create problems for the desired estimates, since we only need a bound for $\mathcal{F}^1(u)$ from above. Moreover it will be a useful information when dealing with general $\alpha$-compatible domains.

Let us now study the second term of $\mathcal{F}(u)$, namely $\mathcal{F}^{2}_{ij}(u)$. Without loss of generality, we can assume that $a=0$ and $h(b)=0$.
We will use the steps of the covering of $T$ described in Remark \ref{rem:vitaliSEE}.
Since $h(x_1)\leq h(0)-c_2x_1$, the side $l_0$ of $Q^0_T$ can be estimated as a fraction of the "height" $h(0)$ of $T$ by 
$$
l_0\leq r_H h(0), \,\,\,\,\,\,\,\,\,r_H=\frac{1}{c_2+1}\,.
$$ 
As well, since $h(x_1)\leq (b-x_1)c_1$, the side $l_0$ of $Q^0_T$ can be estimated as a fraction of the "basis" $b$ of $T$ by
$$
l_0\leq r_B b, \,\,\,\,\,\,\,\,\,r_B=\frac{1}{1+\frac{1}{c_1}}\,.
$$ 
At the second step, we consider the squares $Q^{1,u}_T$ and 
$Q^{1,r}_T$.
With the same arguments of the first step, we can say that the side of $Q^{1,r}_T$ will be  a fraction of the height $l_0$
of the triangular domain in $T \setminus Q^0_T$
on the right of $Q^0_T$. 
At the same time, the side of $Q^{1,u}_T$ will be  a fraction of the basis $l_0$
of the triangular domain in $T \setminus Q^0_T$
on the top of  $Q^0_T$. 
Therefore  
the sides of  $Q^{1,u}_T$ and $Q^{1,r}_T$ are smaller than
$$
\max\{h(0),b\}[\max\{r_B,r_H\}]^2\,.
$$
In general, at the $n^{th}$ step, we will add $2^n$ squares whose sides are smaller than 
$$
\max\{h(0),b\}[\max\{r_B,r_H\}]^n\,.
$$ 
This and (\ref{eq:stimaquadratodilatatoDMP}) allow us to give the following estimate for $\mathcal{F}^{2}_{ij}(u)$
\[
\mathcal{F}^{2}_{ij}(u) \leq \max\{h(0),b\}^{\alpha+1}\sum_{n=0}^{\infty} [2[\max\{r_B,r_H\}]^{\alpha+1}]^{n}\,.
\]
The last series is finite under the assumption 
$2[\max\{r_B,r_H\}]^{\alpha+1}<1$ and this is satisfied according to Definition \ref{defn:compatible-triang}.

\begin{rem}
In the simple case where $T=T_h$ with $h(x_1)=1-x_1$, the assumption
$2\left[\max\{r_B,r_H\}\right]^{\alpha+1}<1$ is equivalent to $\left(\frac 12\right)^{\alpha+1}\leq \frac 12$,
which is satisfied for any $\alpha>0$.
\end{rem}

\subsection{Proof of Theorem \ref{thm:main}: the general case}
\label{subsec:ProofMainThm}

The definition of the candidate solution $u$ in a general $\alpha$-compatible domain $\Omega$ is straightforward once we have performed the previous constructions. Indeed let $\Omega$ be the union of $I$ rectangles $\{R_i\}_{i=1}^I$ and $J$ rotations of $\alpha$-compatible triangular domains $R_j[T_{h_j}]$ for $j\in {1,\dots,J}$, and define $u$ using in any rectangle $R_i$ the construction described in Subsection \ref{subsec:rectangular} and in any $R_j[T_{h_j}]$ the one described in Subsection \ref{subsec:triangular}. The desired bound on the term $\mathcal{F}^1(u)$ easily follows by the previous analysis once we recall that $\Sigma^u_\infty\cap \overline R_i$ has finite $\Haus^1$-measure for any $i$ and that for any point $x\in R_j[T_{h_j}]$ we have $\dist(x,\partial \Omega) \leq \dist\left(x,R_j[G(h_j)]\right)$ where $G(h_j)$ is the graph of the function $h_j$. The bound on the second term follows adding up the estimates for each rectangle and each $\alpha$-compatible triangular domain.

\bibliographystyle{plain}
\end{document}